\documentclass[12pt]{article}
\usepackage[english]{babel}
\usepackage{amsmath, amssymb}
\usepackage{mathtools}
\usepackage{amsthm} 
\usepackage{url} 
\usepackage{makeidx} 
\usepackage{array}
\usepackage{float,placeins}
\usepackage{wrapfig}
\usepackage{graphicx}
\usepackage{rotating}
\usepackage{nicefrac}
\usepackage{longtable}
\usepackage{subfiles,thm-restate}

\usepackage[dvipsnames]{xcolor}

\usepackage{nicefrac}

\usepackage{fullpage}

\usepackage{hyperref}
\usepackage{cleveref}

\newtheoremstyle{plainsl}%
	{\topsep}
	{\topsep}
	{\slshape} 
	{}
	{\normalfont\bfseries}
	{.}
	{ }
	{}

{\theoremstyle{plainsl}
\newtheorem{Theorem}{Theorem}[section]

\newtheorem*{Example}{Example}

\newtheorem{Lemma}[Theorem]{Lemma}
\newtheorem{Cor}[Theorem]{Corollary}

\newtheorem{*Not}[Theorem]{Notation}

\newtheorem{Con}[Theorem]{Conjecture}
\newtheorem{Prob}[Theorem]{Open Problem}
}

{\theoremstyle{remark}
}

\newtheoremstyle{plainnormal}%
	{\topsep}
	{\topsep}
	{\normalfont}
	{}
	{\normalfont\bfseries}
	{.}
	{ }
	{}

\theoremstyle{plainnormal}
\newtheorem*{Def}{Definition}
\newcommand\dir[1]{\mkern2mu\overrightarrow{\mkern-2mu\smash{#1}\vphantom{\textsc{f}}}\mkern2mu}

\usepackage{bbm}
\renewcommand{\mathbb}{\mathbbm}

\def\sqr#1#2{{\vbox{\hrule height.#2pt
    \hbox{\vrule width.#2pt height#1pt \kern#1pt
        \vrule width.#2pt}\hrule height.#2pt}}}
\def\eqed{\sqr{7}{7}}
\def\qedd{%
    \ifmmode\eqno\eqed
    \else\nobreak\ \hfill\eqed\medbreak\fi}

\numberwithin{equation}{section}

\definecolor{vcolour}{RGB}{230,97,0}
\definecolor{rcolour}{RGB}{93,58,155}
\definecolor{gcolour}{RGB}{0,120,207}
\definecolor{jcolour}{RGB}{93,198,155}

\renewcommand{\phi}{\varphi}

\newcommand\gab[1]{{\bf \textcolor{gcolour}{#1  - Gabriëlle} }}

\newcommand\jo[1]{}

\newcommand{\NGi}[4][]{%
  N^{#1}_{#2}(#3,#4)%
}
\usepackage[normalem]{ulem}

\usepackage{tikz}
\usetikzlibrary{calc,decorations.markings,arrows.meta}
\usepackage{subfiles}
\usepackage{subcaption}
\usepackage{enumitem}
\usepackage{comment}

\makeindex

\author{Krystal Guo, Ross J. Kang and Gabriëlle Zwaneveld}

\begin{document}

\title{Seymour-tight orientations} 
\date{1 April 2026} 
\maketitle

\begingroup
    \renewcommand\thefootnote{}
    \footnotetext[0]{Korteweg-de Vries Institute for Mathematics, University of Amsterdam, Amsterdam, The Netherlands. (\texttt{\{k.guo,r.kang,g.c.zwaneveld\}@uva.nl})  }
\endgroup


\begin{abstract}
We investigate `almost counterexamples' to Seymour's second neighbourhood conjecture. In what we call \textsl{Seymour-tight orientations}, the size of the first neighbourhood of each vertex equals the size of its second neighbourhood. We give several examples and constructions. Specifically, we prove that the class of Seymour-tight orientations is closed under taking (generalized) lexicographic products. Moreover, the lexicographic product of a putative counterexample to Seymour's second neighbourhood conjecture and a Seymour-tight orientation is again a counterexample. 

Using lexicographic products, we show that if the conjecture is false, then there exist counterexamples that are close to regular tournaments, and moreover that any digraph occurs as an induced subgraph of a counterexample. We then use this same machinery to construct special putative counterexamples to Sullivan's conjecture.

The inherent symmetry of these orientations give access to an algebraic perspective. Seymour-tight orientations that are also Cayley digraphs correspond to special pairs of critical sets in groups, which connects potentially to additive combinatorics. We use Kemperman's theorem to characterize those Seymour-tight orientations that are the Cayley digraph of an abelian group.


\end{abstract}

\section{Introduction}

We introduce Seymour-tight orientations, motivated by the equality case in Seymour’s second neighbourhood conjecture and as a natural symmetry condition for oriented graphs. We proceed with some definitions. Let $G$ be a directed graph. The \textsl{out-neighbourhood} of a vertex $v \in V(G)$ is the set $\NGi{1}{G}{v} := \{ w \in V(G) \mid (v,w) \in E(G)\}$. It consists precisely of the vertices $w$ for which there is an arc from $v$ to $w$. 
The \textsl{second out-neighbourhood} of $v$ is \[\NGi{2}{G}{v} := \left\{ w\in V(G) - (\NGi{1}{G}{v} \cup \{v\}) \, \middle\vert \, \exists u \in \NGi{1}{G}{v} \text{ s.t. } (u,w) \in E(G)\right\}.\] Thus, it consists of all vertices that can be reached in two steps, but not less, from $v$. 
An \textsl{orientation} is a digraph that contains no cycles of length 2. Equivalently, it can be seen as an assignment of a direction to each edge of an undirected graph. The following is known as  Seymour's second neighbourhood conjecture. 

\begin{Con}[Seymour~\cite{dean1995squaring}]\label{conj:Seymour}
    Every orientation $G$ contains at least one vertex $v$ such that $|\NGi{2}{G}{v}| \geq |\NGi{1}{G}{v}|$.
\end{Con}

This open conjecture, which is related to the Caccetta-H\"{a}ggkvist conjecture~\cite{CaccettaHaggkvist1978}, has received considerable attention~\cite{CHEN2023272,cohn2016number,DAAMOUCH2021332,fidler2007remarks,halkiewicz2026seymour}. It is known to hold for several special classes of digraphs. In particular, it was solved for tournaments by Fischer~\cite{fisher1996squaring} (see also an alternative proof by Havet and Thomassé~\cite{havet2000median}), confirming a conjecture of Dean~\cite{dean1995squaring}. 

Chen, Shen and Yuster~\cite{chen2003second} proved that every orientation $G$ has a vertex $v$ satisfying $|\NGi{2}{G}{v}| \geq \gamma |\NGi{1}{G}{v}|$ where $\gamma =  0.657298\ldots$, the unique real root of $2x^3+x^2-1=0$. This constant $\gamma$ was recently improved by Huang and Peng to 0.715538~\cite{huang2024improved}. Espuny D\'iaz, Gir\~ao, Granet and Kronenberg~\cite{espuny2025seymour} have shown for all $p<\frac{1}{2}$ that the conjecture holds for all orientations of $G(n,p)$ a.a.s.~(with probability tending to 1 as $n \rightarrow \infty$). Moreover, they proved that if the conjecture is false then for all $p \in (\frac{1}{2},1)$, there exists a.a.s.~an orientation of $G(n,p)$ which is a counterexample. 

A counterexample $G$ to Seymour's second neighbourhood conjecture satisfies $|\NGi{1}{G}{v}| > |\NGi{2}{G}{v}|$ for all $v \in V(G)$; properties of (minimal) counterexamples have been studied.
For example, the minimum out-degree of a counterexample is at least $7$~\cite{kaneko2001minimum} and  satisfies $\delta^{+}(G) \geq \sqrt{|V(G)|}$~\cite{espuny2025seymour}. Further, if there exists a counterexample with minimum out-degree $\delta$, then there exists a counterexample on at most $\binom{\delta+1}{2}$ vertices~\cite{seacrest2018seymour}. 



Rather than putative counterexamples, what are the properties of graphs that are \textsl{close} to counterexamples? We are interested in orientations that satisfy $|\NGi{1}{G}{v}| \geq |\NGi{2}{G}{v}|$ for all $v \in V(G)$. Note that by adding a universal sink $s$, we get an orientation where $|\NGi{1}{G}{s}| =  |\NGi{2}{G}{s}|=0$ and $|\NGi{1}{G}{v}| >  |\NGi{2}{G}{v}|$ for all $v \neq s$; loosely speaking, we may think of these as being as close to a counterexample as possible. For convenience,  we call orientations satisfying $|\NGi{1}{G}{v}| \geq |\NGi{2}{G}{v}|$ for all $v \in V(G)$  \textsl{Seymour orientations}.  

Within this class of orientations, we mostly focus on those that are tightest with respect to the condition, which we might consider a type of symmetry. We say that an orientation $G$ is \textsl{Seymour-tight} if $|\NGi{1}{G}{v}| =  |\NGi{2}{G}{v}|$ for all $v \in V(G)$. Focusing on symmetric graphs has often been a useful strategy to construct extremal examples and/or counterexamples. 

Seymour-tight orientations give us a better understanding of how counterexamples to Seymour's second neighbourhood conjecture may arise. We will see that the class of Seymour-tight orientations are closed under taking lexicographic products. Based on this, we observe that  the lexicographic product of a counterexample and a Seymour orientation, in either order, 
is a  counterexample~\cite{brantner2009contributions}. 
Among other applications, we use this observation to show that there are counterexamples (if one exists) that are close to (regular) tournaments (Corollary~\ref{Close to reg tournament}), and to show the following result.

\begin{restatable}{Theorem}{Theoreminduced}
     If Seymour's second neighbourhood conjecture is false, then every orientation $D$ is an induced subgraph of a strongly connected counterexample.
\end{restatable}
\vspace{-10 pt}
\begin{figure}[ht]
    \centering
    \begin{tikzpicture}[scale=1, every node/.style={circle, draw, inner sep=1.5pt}]

\tikzstyle{edge} = [
    postaction={
      decorate,
      decoration={
        markings,
        mark=at position 0.6 with {\arrow[scale=1.2]{>}}}}]

\tikzstyle{vertex}=[circle, draw, fill=black,
                        inner sep=0pt, minimum width=4pt]
\tikzset{
    arrow/.style={
        -{Latex[length=5mm,width=3mm]},  
        thick,                             
        line width=1.5pt                   
    }
}


\node[vertex] (C1) at (0,1.5) {};
\node[vertex] (C2) at (0,-1.5) {};
\node[vertex] (C3) at (2.598,0) {};

\draw[edge] (C1)--(C2);
\draw[edge] (C2)--(C3);
\draw[edge] (C3)--(C1);

\draw[arrow] (1,1.4) .. controls (1,2.5) and (3,3) .. (4,3);
\draw[arrow] (1,-1.4) .. controls (1,-2.5) and (3,-3) .. (4,-3);

\begin{scope}[shift={(6.5,2.5)}, scale=1.2]

\node[vertex] (A1a) at (90:1.2) {};
\node[vertex] (A1b) at (90:2) {};
\node[vertex] (A2a) at (210:1.2) {};
\node[vertex] (A2b) at (210:2) {};
\node[vertex] (A3a) at (330:1.2) {};
\node[vertex] (A3b) at (330:2) {};

\draw[edge] (A1a)--(A2a);
\draw[edge] (A2a)--(A3a);
\draw[edge] (A3a)--(A1a);
\draw[edge] (A1b)--(A2b);
\draw[edge] (A2b)--(A3b);
\draw[edge] (A3b)--(A1b);

\draw[edge] (A1a)--(A2b);
\draw[edge] (A2a)--(A3b);
\draw[edge] (A3a)--(A1b);
\draw[edge] (A1b)--(A2a);
\draw[edge] (A2b)--(A3a);
\draw[edge] (A3b)--(A1a);

\node[draw=none] at (0,-2) {(1)};

\end{scope}


\begin{scope}[shift={(12,-3.3)}, scale=1.2]

\node[vertex] (B1) at (90:0.9) {};
\node[vertex] (B2) at (210:0.9) {};
\node[vertex] (B3) at (330:0.9) {};
\node[vertex] (B4) at (90:2) {};
\node[vertex] (B5) at (210:2) {};
\node[vertex] (B6) at (330:2) {};

\draw[edge] (B1)--(B2);
\draw[edge] (B2)--(B3);
\draw[edge] (B3)--(B1);
\draw[edge] (B4)--(B5);
\draw[edge] (B6)--(B4);
\draw[edge] (B5)--(B6);

\draw[edge] (B1)--(B5);
\draw[edge] (B2)--(B6);
\draw[edge] (B3)--(B4);

\node[draw=none] at (0,-1.565) {(4)};

\end{scope}


\begin{scope}[shift={(6.3,-1.5)}, scale=1.2]

\node[vertex] (P1) at (270:1) {};
\node[vertex] (P2) at (150:1) {};
\node[vertex] (P3) at (30:1) {};
\node[vertex] (P4) at (270:2.35) {};

\draw[edge] (P1)--(P3);
\draw[edge] (P2)--(P1);
\draw[edge] (P3)--(P2);
\draw[edge] (P4)--(P1);

\node[draw=none] at (0,-3) {(3)};

\end{scope}


\begin{scope}[shift={(12.5,2.5)}]

\def\Rbig{2.2}
\def\Rsmall{0.7}

\foreach \j in {0,1,2} {
  \node[vertex] (a\j) at ($(90:\Rbig)+({60+120*\j}:\Rsmall)$) {};
}
\foreach \j in {0,1,2} {
  \node[vertex] (b\j) at ($(210:\Rbig)+({60+120*\j}:\Rsmall)$) {};
}
\foreach \j in {0,1,2} {
  \node[vertex] (c\j) at ($(330:\Rbig)+({60+120*\j}:\Rsmall)$) {};
}

\foreach \u in {a0,a1,a2} \foreach \v in {b0,b1,b2} \draw[edge] (\u)--(\v);
\foreach \u in {a0,a1,a2} \foreach \v in {c0,c1,c2} \draw[edge] (\v)--(\u);
\foreach \u in {b0,b1,b2} \foreach \v in {c0,c1,c2} \draw[edge] (\u)--(\v);

\foreach \x/\y in {c0/c1,c1/c2,c2/c0}
  \draw[edge] (\x)--(\y);

\node[draw=none] at (0.3,-2.5) {(2)};

\end{scope}

\end{tikzpicture}
    \caption{Four different ways to construct a larger Seymour-tight orientation from $\dir{C_3}$.\\
    (1) Take a lexicographic product $\dir{C_3}[E_2]$ (Lemma~\ref{Seymour lex}).\\
    (2) Take a generalized lexicographic product $\dir{C_3}[\dir{C_3},E_3,E_3]$ (Corollary~\ref{vervang alle punten}).\\
    (3) Add a source $s$ such that $\NGi{1}{G}{s}=\NGi{1}{G}{v}$ for some $v \in V(G)$ (Lemma~\ref{Source copy neighbourhood}).\\
    (4) Use a digraph homomorphism $G \rightarrow H$ to add a source component $G$ to $H$ (Lemma~\ref{Graph hom construction}).\\
    Note that (1) and (2) are strongly connected, whereas (3) and (4) are not.
    }\label{Figure intro}
\end{figure}
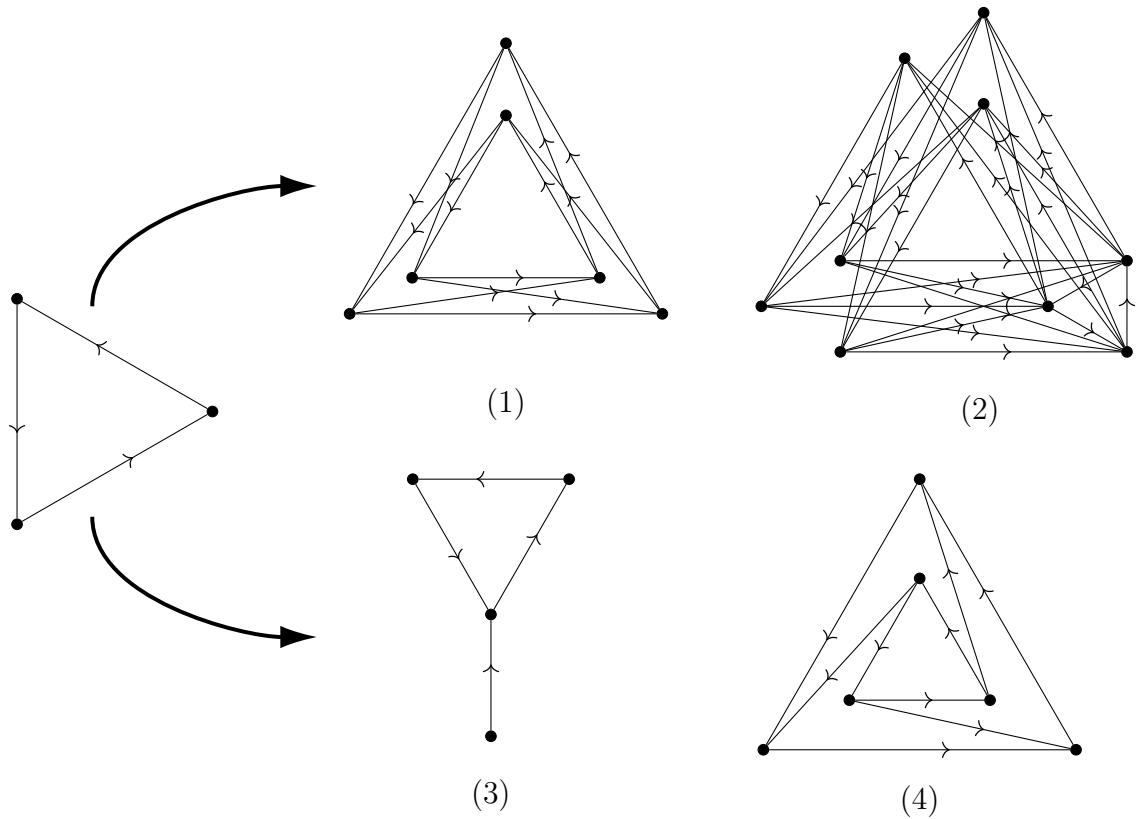 

We also present other ways to build up Seymour-tight orientations from other Seymour-tight orientations, see Figure~\ref{Figure intro}. In Section~\ref{Sec: gen lex}, we show that we can replace subgraphs on which all other vertices are uniform. Hence, the class of Seymour-tight orientations is closed under taking generalized lexicographic products. In Section~\ref{Sec: Strongly disconnected}, we describe how to add a source component to a Seymour-tight orientation to obtain a strongly disconnected Seymour-tight orientation. In Section~\ref{Section: Sullivans conjecture}, we apply our methods to construct Sullivan-tight orientations and special putative counterexamples to Sullivan's conjecture.

Any vertex-transitive Seymour orientation is either a counterexample or a Seymour-tight orientation. It is therefore natural to ask about Seymour(-tight) orientations with more symmetry; it is a classic result of Hamidoune \cite{hamidoune1981application} that there are no Cayley counterexamples. We observe in Section~\ref{Section: Cayley Seymour} that the connection sets of Seymour Cayley orientations correspond to critical pairs of sets in groups, giving a nice connection to structural additive combinatorics~\cite{Grynkiewicz2013}. We use a result of Kemperman~\cite{kemperman1960small}, about critical pairs in abelian groups, to classify all Seymour abelian Cayley orientations. 

\begin{restatable}{Theorem}{ClassificationAbelian}\label{classificiation abelian}
If a Seymour orientation is the Cayley digraph of an abelian group, then it can be constructed by taking (possibly repeated) lexicographic products of empty graphs, the $k$-th power of directed cycles, and of regular tournaments.
\end{restatable}

While the hypothesis is algebraic, the conclusion is purely combinatorial. It is natural to wonder if the same combinatorial classification extends to Seymour orientations that are Cayley digraphs or that are vertex-transitive. 
In the concluding section of the paper, we discuss this and other open problems. In particular, we pose some problems related to conditions for Seymour-tight orientations such that its converse is also Seymour-tight.

\section{Examples and basic properties}\label{Sec: examples}
We start by giving a few examples of Seymour-tight orientations. Note that any directed cycle $\dir{C_n}$, $n \geq 3$, is a Seymour-tight orientation, as every vertex has one vertex in its out-neighbourhood and one vertex in its second out-neighbourhood. In \textsl{the $k$-th power} of a directed graph  $D$, denoted by $D^k$, there is an arc from $v$ to $w$ if and only if there is a path of length at most $k$ from $v$ to $w$ in $D$.

\begin{Lemma}\label{lem: kth power dicycle}
     Let $\dir{C_n}$ be a directed cycle. If $2k <n$, then the $k$-th power of $\dir{C_n}$, denoted by $\dir{C_n}^k$, is a Seymour-tight orientation.
\end{Lemma}
\begin{proof}
    Let $v_i$ be a vertex in the $k$-th power $\dir{C_n}^k$. Then $\NGi{1}{\dir{C_n}^k}{v_i} = \{v_{i+1}, \ldots, v_{i+k}\}$ where the indices are taken modulo $n$. Therefore, the vertices that can be reached in at most two steps from $v_i$ are precisely the vertices $v_{i+1}, \ldots, v_{i+2k}$. Since $2k<n$, we obtain $\NGi{2}{\dir{C_n}^k}{v_i} = \{v_{i+k+1}, \ldots, v_{i+2k}\}$, which implies $|\NGi{1}{\dir{C_n}^k}{v_i}| = k = |\NGi{2}{\dir{C_n}^k}{v_i}|$ for all vertices $v_i$.
\end{proof}

The $k$-th power of a directed cycle of length $n$ is a Cayley digraph of $\mathbb{Z}_n$ with connection set $\{1, \ldots,k \}$. Next, we characterize which tournaments are Seymour-tight orientations.

\begin{Lemma}
    A tournament is a Seymour-tight orientation if and only if it is regular.
\end{Lemma}
\begin{proof}
    Let $T$ be a tournament on $n$ vertices that is also a Seymour-tight orientation. If there exists a vertex $v$ with $|\NGi{1}{T}{v}| > \frac{n-1}{2}$, then 
    \[|\NGi{2}{T}{v}| \leq V(T)- |\{v\}| - |\NGi{1}{T}{v}|< n-1-\frac{n-1}{2}=\frac{n-1}{2}=|\NGi{1}{T}{v}|.\]
    Thus $T$ is not a Seymour-tight orientation.
    Hence, all vertices $v$ satisfy $|\NGi{1}{T}{v})| \leq (n-1)/2$. In a tournament the average out-degree equals $(n-1)/2$, thus $|\NGi{1}{T}{v}| = (n-1)/2$ for all $v$, implying that $T$ must be regular. 

    We now prove that any regular tournament $T$ is Seymour-tight. We start by showing that the diameter of $T$ is $2$. Suppose for a contradiction that $v$ and $w$ are two vertices such that $w$ cannot be reached within two steps from $v$. Then for all vertices $u \in \NGi{1}{T}{v}$, the edge $\{u,w\}$ is oriented from $u$ to $w$. Therefore, $w$ must have at most \[n-1-|\NGi{1}{T}{v}| = n -1- \frac{n+1}{2} = \frac{n-3}{2}\] out-neighbours, a contradiction. Since $T$ has diameter $2$, \[|\NGi{2}{T}{v}| = n-1-|\NGi{1}{T}{v}| = n-1-\frac{n-1}{2} = \frac{n-1}{2} = |\NGi{1}{T}{v}|\]
    for all vertices $v$. Thus, $T$ is a Seymour-tight orientation.
\end{proof}

Two vertices $u,v \in V(D)$ lie in the same \textsl{strongly connected component} of a directed graph $D$ if there exists a directed walk from $u$ to $v$ and a directed walk from $v$ to $u$. A directed graph $D$ is \textsl{strongly connected} if it has only one strongly connected component. Every directed graph can be partitioned into its strongly connected components $A', \ldots, A_k$. The \textsl{condensation of $D$} is the graph where every strongly connected component is contracted into one vertex. There is an arc $A_i \rightarrow A_j$ in the condensation if and only if there exist $v \in A_i$ and $w \in A_j$ such that $v \rightarrow w$ is an arc in $D$. 

By definition, the condensation of $D$ is a directed acyclic graph. Let $\mathcal{A}_k$ be the set of $A_i$ that can be reached from $A_k$ in the condensation of $D$. If $D$ is a Seymour-tight orientation, then $\mathcal{A}_k$ is also a Seymour-tight orientation for all $k$. In particular, all $A_i$ that are sink vertices in the condensation of $D$ are Seymour-tight orientations. Hence, every strongly disconnected Seymour-tight orientation can be formed by starting with a strongly connected Seymour-tight orientation and then adding `source' components to it, see Section~\ref{Sec: Strongly disconnected}.

In Sections~\ref{Section: Lex prod} and~\ref{Sec: gen lex}, we show that the class of Seymour-tight orientations is closed under taking (generalized) lexicographic products. 
If $D$ is strongly connected, then $D[G_1, \ldots, G_k]$ is also strongly connected even if (some of) the $G_i$'s are not. Hence, we can construct a strongly connected Seymour-tight orientations from strongly disconnected Seymour-tight orientations $G_1,\ldots, G_k$ each on $n$ vertices by taking a generalized lexicographic product $D[G_1,\ldots,G_k]$ where $D$ is a strongly connected Seymour-tight orientation on $k$ vertices.


\section{Lexicographic products}\label{Section: Lex prod}
Let $D$ and $G$ be two directed graphs, then the \textit{lexicographic product} of $D$ and $G$, denoted $D[G]$ satisfies $V(D[G])=V(D) \times V(G)$. There is a directed edge from $(v,i)$ to $(w,j)$ if and only if there is a directed edge from $v$ to $w$ in $D$ or $v=w$ and there is a directed edge from $i$ to $j$ in $G$~\cite{bang2018classes}. Hence, the lexicographic product of two orientations is again an orientation. Moreover, the underlying graph of $D[G]$ is the lexicographic product of the underlying graph of $D$ with the underlying graph of $G$.

\begin{Lemma}\label{Seymour lex}
Let $D$ and $G$ be two Seymour-tight orientations. Then the lexicographic product $D[G]$ is also a Seymour-tight orientation. Moreover, if $D$ is a strongly connected Seymour-tight orientation, then $D[G]$ is also strongly connected.
\end{Lemma}
\begin{proof}
Let $D,G$ be two Seymour-tight orientations. Let $(v,i)$ be a vertex in $D[G]$. Then
    \[\NGi{1}{D[G]}{(v,i)} = \left\{(w,j) \, \middle\vert \, w \in \NGi{1}{D}{v} \text{ or } v=w \text{ and } j \in \NGi{1}{G}{i}\right\}.\]
    In particular, $|\NGi{1}{D[G]}{(v,i)}| = |V(G)| \cdot |\NGi{1}{D}{v}| + |\NGi{1}{G}{i}|$.

    All vertices $(w,j)$ that can be reached from $(v,i)$ in at most two steps satisfy either $w \in \NGi{1}{D}{v} \, \cup \, \NGi{2}{D}{v}$ or $w=v$ and $j \in \NGi{1}{G}{i} \, \cup \, \NGi{2}{G}{i}$. By deleting those in $\NGi{1}{D[G]}{(v,i)}$, we obtain
    \[\NGi{2}{D[G]}{(v,i)} = \left\{(w,j) \, \middle\vert \, w \in \NGi{2}{D}{v} \text{ or } v=w \text{ and } j \in \NGi{2}{G}{i}\right\},\]
    thus implying $|\NGi{2}{D[G]}{(v,i)}| = |V(G)| \cdot |\NGi{2}{D}{v}| + |\NGi{2}{G}{i}|$. Since $D$ and $G$ are Seymour-tight orientations, we have
    \begin{align*} |\NGi{1}{D[G]}{(v,i)}| &= |V(G)| \cdot |\NGi{1}{D}{v}| + |\NGi{1}{G}{i}|\\&= |V(G)| \cdot |\NGi{2}{D}{v}| + |\NGi{2}{G}{i}|=  |\NGi{2}{D[G]}{(v,i)}| \end{align*}
    for all vertices $(v,i)$. Hence, we obtain that $D[G]$ is also a Seymour-tight orientation.

    Moreover, by definition of the lexicographic product of directed graphs, we obtain that if $D$ is a strongly connected, then also $D[G]$ is strongly connected.
\end{proof}

\subsection{Putative counterexamples}\label{Section: Putative counterexamples}

With the lexicographic product, we can not only construct new Seymour-tight orientations, but also obtain counterexamples from smaller ones. Using similar arguments as in the proof of Lemma~\ref{Seymour lex}, we can show the following.
\begin{Theorem}[Theorem 4.2~\cite{brantner2009contributions}]\label{Lex product counterexample}
If $O$ is a counterexample to Seymour's second neighbourhood conjecture and $G$ is a Seymour orientation, then $O[G]$ and $G[O]$ are counterexamples.
\end{Theorem}

In particular, we can take $G$ to be Seymour-tight. By choosing $G$ appropriately, we can construct counterexamples that satisfy some interesting properties. 

\begin{Cor}\label{Close to reg tournament} If Seymour's second neighbourhood conjecture is false, then there exists $k \in \mathbb{N}$ such that there are infinitely many counterexamples $O$ whose minimum out-degree is at least $\frac{V(G)}{2}-k$.
\end{Cor}
\begin{proof}
    Let $O$ be a counterexample to Seymour's second neighbourhood conjecture and set $k = |V(O)|$. Let $T$ be a regular tournament on $m$ vertices. By Theorem \ref{Lex product counterexample}, $T[O]$ is also a counterexample to Seymour's second neighbourhood conjecture. The number of vertices of this graph is $m|V(O)|$, while  its minimum out-degree is $\frac{m-1}{2}|V(O)|= \frac{m|V(O)|}{2}-\frac{|V(O)|}{2}=\frac{V(T[O])}{2}-\frac{k}{2}$. Since there are infinitely many regular tournaments, the statement holds.
\end{proof}

Similarly, by taking $G=\dir{C_n}$, we observe that there exists $k \in \mathbb{N}$ such that there are infinitely many counterexamples of maximum degree at most $k$~\cite{espuny2025seymour}. By repeatedly taking the lexicographic product of counterexamples, we obtain the following result.

\begin{Lemma}\label{lem: sequence of counterexamples}
    If Seymour's second neighbourhood conjecture is false, then there exists $\epsilon > 0$ such that for all $k \in \mathbb{N}$, there exists a strongly connected orientation $O$ with minimum out-degree at least $k$ such that $(1-\epsilon) |\NGi{1}{O}{v}| \geq |\NGi{2}{O}{v}|$ for all $v \in V(O)$.
\end{Lemma}
\begin{proof}
    Let $O$ be a minimal counterexample to Seymour's second neighbourhood conjecture. Suppose $O$ has $n$ vertices and denote its maximal out-degree with $\Delta$ and its minimum out-degree with $\delta$. Every vertex of out-degree $d$ has at most $d-1$ vertices in its second neighbourhood.  Set $\epsilon = \frac{1}{\Delta}$, then $(1-\epsilon)d = \frac{\Delta-1}{\Delta} d \geq d-1$, thus  $(1-\epsilon) |\NGi{1}{O}{v}| \geq |\NGi{2}{O}{v}|$ for all $v \in V(O)$. 

    We define a sequence of graphs $O_i$, where $O_1=O$ and $O_{k+1}=O[O_k]$ for all $k$. Suppose that $(1-\epsilon) |\NGi{1}{O_i}{v}| \geq |\NGi{2}{O_i}{v}|$ for all $v \in O_i$ and $i \in \{1,\ldots,k\}$. Let $(v,w) \in O_{k+1}=O[O_k]$ such that $v \in O$ and $w \in O_k$. Then by definition of the lexicographic product
    \begin{align*}|\NGi{1}{O_{k+1}}{(v,w)}| &= |\NGi{1}{O}{v}| \cdot |V(O_k)|+|\NGi{1}{O_k}{w}|.\\
    |\NGi{2}{O_{k+1}}{(v,w)}| &= |\NGi{2}{O}{v}| \cdot |V(O_k)|+|\NGi{2}{O_k}{w}|.
    \end{align*}
    Therefore,
    \begin{align*}(1-\epsilon)|\NGi{1}{O_{k+1}}{(v,w)}| &= (1-\epsilon)|\NGi{1}{O}{v}| \cdot |V(O_k)|+ (1-\epsilon)|\NGi{1}{O_k}{w}|\\
    &\geq  |\NGi{2}{O}{v}| \cdot |V(O_k)|+|\NGi{2}{O_k}{w}|=|\NGi{2}{O_{k+1}}{(v,w)}| 
    \end{align*}
    Thus, also $O_{k+1}$ satisfies this property. By induction, every graph $O_i$ satisfies this property. Moreover, note that $|V(O_k)|=n^k$ for all $k$ and therefore the minimum degree of $O_k$ is at least $\delta n^{k-1}$ which goes to $\infty$ as $k \rightarrow \infty$.
\end{proof}

So Seymour's second neighbourhood conjecture is equivalent to the following conjecture:

\begin{Con}
    Let $\epsilon >0$ be arbitrary. Then every oriented graph $G$ has at least one vertex satisfying $|\NGi{2}{G}{v}| \geq (1-\epsilon) |\NGi{1}{G}{v}|$.
\end{Con}

\subsection{Induced subgraphs}\label{Sec: induced subgraphs}
In this subsection, we prove that every oriented graph is an induced subgraph of a Seymour-tight orientation.

\begin{Lemma}\label{Every graph is induced subgraph}
    Every oriented graph $D$ is an induced subgraph of a strongly connected Seymour-tight orientation.
\end{Lemma}
\begin{proof}
    Let $D$ be an arbitrary orientation. We construct a new graph $D'$ that contains $D$ as an induced subgraph. We start by adding $|V(D)|$ common sinks to $D$, i.e.~for all new sink vertices $s$ we have $\NGi[-]{1}{D}{s}=V(D)$. Moreover, for every $v \in V(D)$ we add a new sink $s_v$ such that $\NGi[-]{1}{D}{s_v}=v$ to obtain $D'$. Then all vertices $v \in V(D')$ satisfy $|\NGi{1}{D'}{v}|= |\NGi{1}{D}{v}|+|V(D)|+1$ and $|\NGi{2}{D'}{v}| = |\NGi{2}{D}{v}|+ |\NGi{1}{D}{v}|$. In particular, $|\NGi{1}{D'}{v}|-|\NGi{2}{D'}{v}| =: k_v \geq 0$. We now place the sinks $s_v$ in the $k_v$-regular Seymour-tight orientation $\dir{C_3}[E_{k_v}]$, by adding new vertices. By construction, $\NGi{2}{D'}{v}$ becomes $k_v$ larger, while $\NGi{1}{D'}{v}$ stays the same. Thus, in this new graph, we have $|\NGi{1}{D'}{v}| = |\NGi{2}{D'}{v}|$ for all $v \in V(D)$. All common sinks $s$ satisfy $|\NGi{1}{D'}{s}|=0=|\NGi{2}{D'}{s}|$. Lastly, for all vertices $w_v$ in the new Seymour-tight orientation of sink $s_v$, we have $|\NGi{1}{D'}{w_v}|=k_v=|\NGi{2}{D'}{w_v}|$ as these Seymour-tight orientations are sink parts of our graph $D'$. Hence, $D'$ is a Seymour-tight orientation that contains $D$ as an induced subgraph.

    To obtain a strongly connected Seymour-tight orientation we consider the graph $\dir{C_3}[D']$ which is strongly connected and contains $D'$ as an induced subgraph implying that it contains $D$ as an induced subgraph.
\end{proof}

Using this we can prove that if Seymour's second neighbourhood conjecture is false, then any orientation is an induced subgraph of a counterexample.

\Theoreminduced*
\begin{proof}
    Let $O$ be a vertex-minimal connected counterexample to Seymour's second neighbourhood conjecture. Then $O$ is strongly connected. By Lemma \ref{Every graph is induced subgraph}, there exists a Seymour-tight orientation $S$ such that $D$ is an induced subgraph of $S$. Now by Theorem \ref{Lex product counterexample} we find that $O[S]$ is a counterexample to Seymour's second neighbourhood conjecture. Moreover, by definition of the lexicographic product, we see that $S$ is an induced subgraph of $O[S]$ and therefore, also $D$ is an induced subgraph of $O[S]$. Since $O$ is strongly connected, we find that $O[S]$ is also strongly connected. In conclusion, $O[S]$ is a strongly connected counterexample to Seymour's second neighbourhood conjecture that has $D$ as an induced subgraph.
\end{proof}

\section{Generalized lexicographic products}\label{Sec: gen lex}

Let $X \subseteq V(D)$ be a set of vertices in a Seymour-tight orientation $D$. A vertex $v \in D-X$ is \textsl{uniform on X} if one of the following holds:
\begin{enumerate}[label=(\alph*),nosep]
    \item   all vertices in $X$ are in-neighbours of $v$;
    \item all vertices in $X$ are out-neighbours of $v$; or 
    \item no vertex in $X$ is an in- or out-neighbour of $v$.
\end{enumerate}

\begin{Lemma} Let $D$ be an orientation. Suppose that all vertices in $D-X$ are uniform on $X$. Let $v \in D-X$, then either all vertices of $X$ are in the second out-neighbourhood (resp. second in-neighbourhood) of $v$ or no vertex of $X$ in the second out-neighbourhood (resp. second in-neighbourhood) of $v$. 
\end{Lemma}
\begin{proof}
    Suppose that there exists a $x \in X$ such that $x \in \NGi{2}{D}{v}$. Let $x' \in X$, be arbitrary. By definition, $x \not \in \NGi{1}{D}{v}$. Since $v$ is uniform on $X$, we also have that $x' \not \in  \NGi{1}{D}{v}$. Moreover, there exists a $w$ such that there are arcs $v \rightarrow w$ and $w \rightarrow x$. Since $w$ is uniform on $X$, there is also an arc $w \rightarrow x'$. Hence, $x' \in \NGi{2}{D}{v}$ and we can conclude that if $v$ has at least one second out-neighbour that lies in $X$, then every vertex of $X$ is a second out-neighbour of $v$. The proof for second in-neighbourhoods follows analogously.
\end{proof}


So for any $v \in D-X$, we have that either $v$ is in the second out-neighbourhood (resp. second in-neighbourhood) of all vertices in $X$ or in the second out-neighbourhood (resp. second in-neighbourhood) of no vertex in $X$. In particular, the number of vertices in $D-X$ in the first and second neighbourhood of some vertex $x \in X$ does not depend on the particular vertex $x \in X$. If $D$ is a Seymour-tight orientation, there exists $k \in \mathbb{Z}$ such that for all $x \in X$, $|\NGi{1}{D|_X}{x}|-|\NGi{2}{D|_X}{x}|=k$. 

\begin{Lemma}\label{lem: replacement} Let $D$ be a Seymour orientation and let $X \subseteq V(D)$ be such that all vertices in $D-X$ are uniform on $X$. 
Let $D'$ be the orientation for which the induced digraph on the vertices of $X$ is replaced by an induced digraph $H$ on $X$ satisfying 
\[|\NGi{1}{H}{x}|-|\NGi{2}{H}{x}| \geq |\NGi{1}{D|_X}{x}|-|\NGi{2}{D|_X}{x}|\] for all $x \in H$. Then $D'$ is a Seymour orientation. \\
Further, if $D, D|_X$ and $H$ are all Seymour-tight orientations, then $D'$ is also Seymour-tight.
\end{Lemma}
\begin{proof}
Since all vertices in $D-X$ are uniform on $X$, the first and second neighbourhood of all vertices in $D-X$ remain unchanged. For $x\in X$, we have
\begin{align*}|\NGi{1}{D'}{x}|-|\NGi{2}{D'}{x}|&=|\NGi{1}{H}{x}|-|\NGi{2}{H}{x}|+|\NGi{1}{D'-H}{x}|-|\NGi{2}{D'-H}{x}|\\
&\geq |\NGi{1}{D|_X}{x}|-|\NGi{1}{D|_X}{x}|+|\NGi{1}{D-X}{x}|-|\NGi{2}{D-X}{x}|\\
&=|\NGi{1}{D}{x}|-|\NGi{2}{D}{x}|.
\end{align*}
Hence, $D'$ is also a Seymour orientation.

If $D$, $D|_X$ and $H$ are all Seymour-tight, then
\begin{align*}|\NGi{1}{D'}{x}|-|\NGi{2}{D'}{x}|&=|\NGi{1}{H}{x}|-|\NGi{2}{H}{x}|+|\NGi{1}{D'-H}{x}|-|\NGi{2}{D'-H}{x}|\\
&= 0+|\NGi{1}{D'-X}{x}|-|\NGi{2}{D'-X}{x}|\\
&=|\NGi{1}{D|_X}{x}|-|\NGi{1}{D|_X}{x}|+|\NGi{1}{D-X}{x}|-|\NGi{2}{D-X}{x}|\\
&=|\NGi{1}{D}{x}|-|\NGi{2}{D}{x}|=0.
\end{align*}
Hence, $D'$ is also Seymour-tight.
\end{proof}

Let $D[G]$ be the lexicographic product of two Seymour-tight orientations $D$ and $G$. Then for every $v \in D$, define the set $X_v = \{(v,i) \mid i \in V(G)\}$. 

\begin{Lemma}\label{lem: replace lex prod}
    Let $D[G]$ be the lexicographic product of two Seymour-tight orientations $D$ and $G$. Then for every $v$, $D[G]|_X$ is isomorphic to $G$. 
    Replacing the orientation on $X_v$ with another Seymour-tight orientation yields a new Seymour-tight orientation. 
\end{Lemma}
\begin{proof} By Lemma \ref{Seymour lex}, the orientation $D[G]$ is a Seymour orientation. By definition of the lexicographic product, the orientation induced by $X_v$ is isomorphic to $G$ and is thus Seymour-tight. Moreover, every vertex in $D[G]-X_v$ is uniform on $X_v$. Hence, by Lemma \ref{lem: replacement} replacing the orientation on vertex set $X_v$ by another Seymour-tight orientation on $|X_v|$ vertices again yields a Seymour-tight orientation.
\end{proof}

\begin{Def}
Let $D$ be a directed graph, where $V(D)=[n]$, and let $G_1, \ldots, G_n$ be a sequence of directed graphs. Then the \textsl{generalized lexicographic product} $D[G_1, \ldots, G_n]$ is the graph where we replace every vertex $i$ of $D$ with the graph $G_i$. Moreover, there is an edge from $(i,v)$ to $(j,w)$ if and only if there is an edge from $i$ to $j$ in $D$ or $i=j$ and there is an edge from $v$ to $w$ in $G_i$.
\end{Def} 

The generalized lexicographic product is also known as $H$-join. By applying Lemma \ref{lem: replace lex prod} to every set $X_v$ in the lexicographic product $D[G]$, we obtain the following.

\begin{Cor}\label{vervang alle punten}
    Let $D$ be a Seymour-tight orientation on $n$ vertices. Let $G_1,\ldots, G_n$ be a sequence of Seymour-tight orientations on $k$ vertices. Then the orientation $D[G_1, \ldots, G_n]$ is also a Seymour-tight orientation.
\end{Cor}
\begin{proof}
    By Lemma \ref{Seymour lex}, we obtain that the lexicographical product $D[G_1]$ is Seymour-tight. For every $i = 2,\ldots,n$, we sequentially replace the orientation of $X_i$ by the Seymour-tight orientation $G_i$. By Lemma \ref{lem: replace lex prod}, this new orientation is Seymour-tight. 
\end{proof}

Using this, we can construct infinitely many non-regular strongly connected Seymour-tight orientations, since the $G_i$ can be chosen arbitrarily. Taking for example $D=\dir{C_3}$, $G_1=G_2 = E_3$ and  $G_3 = \dir{C_3}$, we get a Seymour-tight orientation on 9 vertices where some vertices have out-degree 3 while others have out-degree 4, see Figure~\ref{fig:non_reg_example}. Or we might have $D=\dir{C_4}$, $G_1 =\dir{C_4}$, $G_2= \dir{C_3} \cup K_1$, $G_3=E_4$ and $G_4$ consists of $\dir{C_3}$ with an extra vertex that has one outgoing arc pointing towards a vertex on this cycle, see Figure~\ref{fig:non_reg2}.

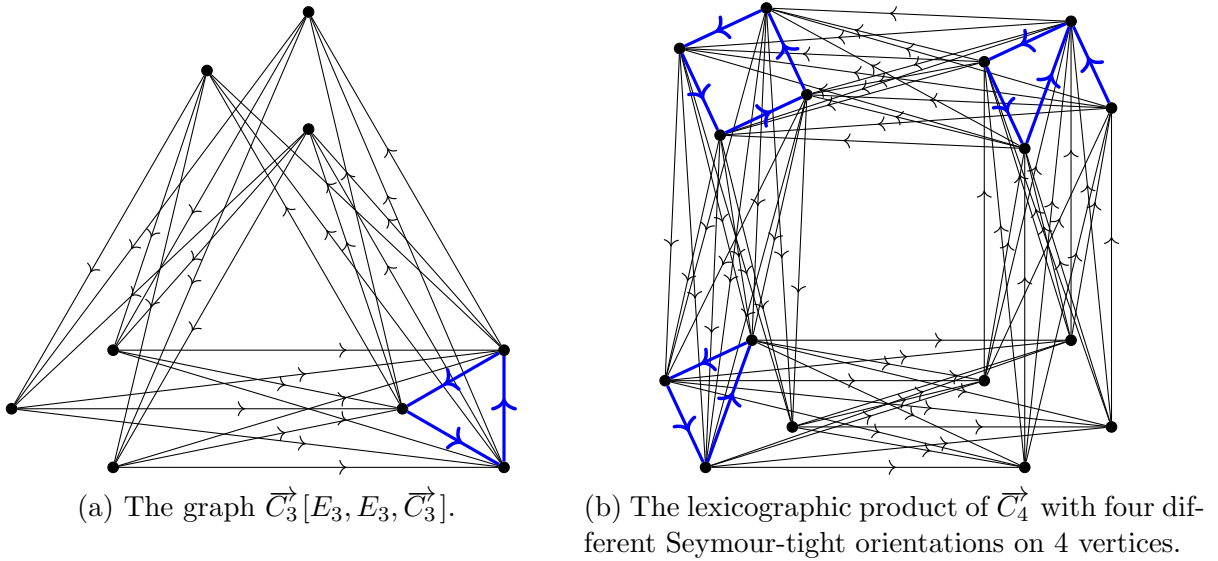
\begin{figure}[ht]
\begin{subfigure}[t]{0.5\textwidth}
    \centering
    \begin{tikzpicture}[scale=1.0]
\tikzstyle{edge} = [
    postaction={
      decorate,
      decoration={
        markings,
        mark=at position 0.6 with {\arrow[scale=1.1]{>}}}}
]
\tikzstyle{vertex}=[circle, draw, fill=black,
                        inner sep=0pt, minimum width=4pt]

\def\Rbig{3}    
\def\Rsmall{0.9}  

\foreach \j in {0,1,2} {
  \node[vertex] (a\j) at ($(90:\Rbig)+({60+120*\j}:\Rsmall)$) {};
}
\foreach \j in {0,1,2} {
  \node[vertex] (b\j) at ($(210:\Rbig)+({60+120*\j}:\Rsmall)$) {};
}
\foreach \j in {0,1,2} {
  \node[vertex] (c\j) at ($(330:\Rbig)+({60+120*\j}:\Rsmall)$) {};
}

\foreach \u in {a0,a1,a2} \foreach \v in {b0,b1,b2} \draw[edge] (\u) -- (\v);
\foreach \u in {a0,a1,a2} \foreach \v in {c0,c1,c2} \draw[edge] (\v) -- (\u);
\foreach \u in {b0,b1,b2} \foreach \v in {c0,c1,c2} \draw[edge] (\u) -- (\v);

\foreach \x/\y in {c0/c1,c2/c0,c1/c2} \draw[edge, color = blue, line width= 1.2 pt] (\x) -- (\y);

\end{tikzpicture}
    \caption{The graph $\dir{C_3}[E_3,E_3,\dir{C_3}]$.}\label{fig:non_reg_example}
\end{subfigure}
\begin{subfigure}[t]{0.5\textwidth}
    \centering
    \begin{tikzpicture}[scale=1.0]
\tikzstyle{edge} = [
    postaction={
      decorate,
      decoration={
        markings,
        mark=at position 0.6 with {\arrow[scale=1.1]{>}}}}
]
\tikzstyle{vertex}=[circle, draw, fill=black,
                        inner sep=0pt, minimum width=4pt]

\def\Rbig{3}    
\def\Rsmall{0.9}  

\foreach \j in {0,1,2,3} {
  \node[vertex] (a\j) at ($(45:\Rbig)+({70+90*\j}:\Rsmall)$) {};
}
\foreach \j in {0,1,2,3} {
  \node[vertex] (b\j) at ($(130:\Rbig)+({70+90*\j}:\Rsmall)$) {};
}
\foreach \j in {0,1,2,3} {
  \node[vertex] (c\j) at ($(-135:\Rbig)+({70+90*\j}:\Rsmall)$) {};
}

\foreach \j in {0,1,2,3} {
  \node[vertex] (d\j) at ($(-45:\Rbig)+({70+90*\j}:\Rsmall)$) {};
}

\foreach \u in {a0,a1,a2,a3} \foreach \v in {b0,b1,b2,b3} \draw[edge] (\u) -- (\v);
\foreach \u in {a0,a1,a2,a3} \foreach \v in {d0,d1,d2,d3} \draw[edge] (\v) -- (\u);
\foreach \u in {b0,b1,b2,b3} \foreach \v in {c0,c1,c2,c3} \draw[edge] (\u) -- (\v);
\foreach \u in {c0,c1,c2,c3} \foreach \v in {d0,d1,d2,d3} \draw[edge] (\u) -- (\v);

\foreach \x/\y in {a0/a1,a2/a0,a1/a2,a3/a0} \draw[edge, color=blue, line width = 1.2 pt] (\x) -- (\y);
\foreach \x/\y in {b0/b1,b3/b0,b1/b2,b2/b3} \draw[edge, color = blue, line width = 1.2 pt] (\x) -- (\y);
\foreach \x/\y in {c0/c1,c2/c0,c1/c2} \draw[edge, color = blue, line width= 1.2 pt] (\x) -- (\y);

\end{tikzpicture}
    \caption{The lexicographic product of $\dir{C_4}$ with four different Seymour-tight orientations on 4 vertices.}\label{fig:non_reg2}
\end{subfigure}
    \caption{Two examples of a strongly connected non-regular Seymour-tight orientations.}\label{fig: not reg example 2x}
\end{figure}

Similarly, we can also extend Theorem \ref{Lex product counterexample} to also hold for generalized lexicographic products. 

\begin{Cor}
    Let $O$ and $O_1,\ldots, O_n$ be counterexamples to Seymour's second neighbourhood conjecture. Let $G$ and $G_1,\ldots, G_n$ be Seymour orientations. Then both of  $G[O_1,\ldots, O_n]$ and $O[G_1,\ldots, G_n]$ are counterexamples  to  Seymour's second neighbourhood conjecture. \qed
\end{Cor}
\begin{proof}
Let $G$ be Seymour-tight and let $O_1,\ldots,O_n$ be counterexamples. Let $(i,v)$ be a vertex in $G[O_1,\ldots,O_n]$. Then
    \[\NGi{1}{G[O_1,\ldots,O_n]}{(i,v)} = \left\{(j,w) \, \middle\vert \, w \in \NGi{1}{G}{j} \text{ or } i=j \text{ and } w \in \NGi{1}{O_i}{v}\right\}.\]
    In particular, since all $O_i$ have the same size, $|\NGi{1}{G[O_1,\ldots,O_n]}{(i,v)}| = |V(O_i)| \cdot |\NGi{1}{G}{i}| + |\NGi{1}{O_i}{v}|$.

    All vertices $(j,w)$ that can be reached from $(i,v)$ in at most two steps satisfy either $j \in \NGi{1}{G}{i} \, \cup \, \NGi{2}{G}{i}$ or $i=j$ and $w \in \NGi{1}{O_i}{v} \, \cup \, \NGi{2}{O_i}{v}$. By deleting those in $\NGi{1}{G[O_1,\ldots,O_n]}{(i,j)}$, we obtain
    \[\NGi{2}{G[O_1,\ldots,O_n]}{(i,v)} = \left\{(j,w) \, \middle\vert \, j \in \NGi{2}{G}{v} \text{ or } i=j \text{ and } w \in \NGi{2}{O_i}{v}\right\},\]
    thus implying $|\NGi{2}{G[O_1,\ldots,O_n]}{(i,v)}| = |V(O_i)| \cdot |\NGi{2}{G}{i}| + |\NGi{2}{O_i}{v}|$. Since $G$ is Seymour and every $O_i$ is a counterexample, we have
    \begin{align*} |\NGi{1}{G[O_1,\ldots,O_n]}{(i,v)}| &= |V(O_i)| \cdot |\NGi{1}{G}{i}| + |\NGi{1}{O_i}{v}|\\&> |V(O_i)| \cdot |\NGi{2}{G}{i}| + |\NGi{2}{O_i}{v}|\\&=  |\NGi{2}{G[O_1,\ldots,O_n]}{(i,v)}| \end{align*}
    for all vertices $(i,v)$. Hence, we obtain that $G[O_1,\ldots,O_n]$ is also a counterexample. Similarly, we can prove that also $O[G_1,\ldots, G_n]$ is a counterexample.
\end{proof}

Note that we could have used this type of proof also to prove Corollary \ref{vervang alle punten}. We can use even more generalized lexicographic-type products to make new Seymour-tight orientations.
Following~\cite{bouya2021seymour}, we define the matrix $S_D$  with entries given by 
\[
S_D(v,w) = \begin{cases}
    1, & \text{ if }w \in N_1(v); \\
    -1, & \text{ if } w \in N_2(v); \\
    0, & \text{ otherwise.}
\end{cases}
\]
By construction, $S_D \mathbf{1}=0$ if and only if $D$ is a Seymour-tight orientation. Moreover, an orientation is a Seymour orientation if and only if $S_D \mathbf{1} \leq 0$. The matrix $S_D^{T}$ corresponds to converse, which is the orientation where all arcs are reversed. We will consider specific vectors in the kernel of $S_D$, which is a subspace of $\mathbb{R}^n$. Since $S_D$ is an integral matrix, there exists a basis of the kernel of $S_D$ with integer entries.

\begin{Theorem}
    Let $D$ be an orientation on $[n]$ and let $\mathbf{x} \in \mathbb{Z}_{\geq 0}$ be a vector such that $S_D\mathbf{x} = 0$. For all $i \in [n]$, let $G_i$ be a Seymour-tight orientation of size $\mathbf{x}_i$. Then the graph $D[(G_i)_{i \in V(D)}]$ is a Seymour-tight orientation.
\end{Theorem}
\begin{proof}
    We can write $S_{D[(G_i)_{i \in [n]}]}$ as a block matrix:
    \[S_{D[(G_i)_{i \in [n]}]} = \begin{pmatrix}
    S_{G_1} & S_D(1,2) \cdot J & S_D(1,3) \cdot J & \ldots & S_D(1,n) \cdot J\\
    S_D(2,1) \cdot J & S_{G_2} & S_D(2,3) \cdot J & \ldots & S_D(2,n) \cdot J\\
    \vdots & \vdots & \vdots & \ddots & \vdots \\
      S_D(n,1) \cdot J & S_D(n,2) \cdot J & S_D(n,3) \cdot J & \ldots & S_{G_n}
    \end{pmatrix},
    \]
    where $J$ is the all ones matrix of the appropriate size. Then $S_{D[(G_v)_{v \in V(D)}]} \mathbf{1}=0$ since $S_{G_i} \mathbf{1}=0$ for all $i$ and $S_D(i,j) \cdot J \mathbf{1}=S_D(i,j) \cdot x_j$. Since $S_D\mathbf{x} = 0$, we have $\sum_{j \neq i} S_D(i,j) \cdot x_j=0$ for all $i$.
\end{proof}

With this result we can even take generalized lexicographic products of some graphs that are not all the same size. 

\begin{figure}[ht]
    \centering
    \begin{tikzpicture}[scale=1]
\tikzstyle{edge} = [
    postaction={
      decorate,
      decoration={
        markings,
        mark=at position 0.6 with {\arrow[scale=1.2]{>}}}}
]
\tikzstyle{vertex}=[circle, draw, fill=black,
                        inner sep=0pt, minimum width=4pt]

\def\Rbig{3.8}    
\def\Rsmall{0.9}  

\foreach \j in {0,1,2} {
  \node[vertex, fill= blue,draw=blue] (a\j) at ($(90:\Rbig)+({60+120*\j}:\Rsmall)$) {};
}
\foreach \j in {0,1,2} {
  \node[vertex, fill=blue, draw=blue] (b\j) at ($(210:\Rbig)+({60+120*\j}:\Rsmall)$) {};
}
\foreach \j in {0,1,2} {
  \node[vertex, fill=blue, draw=blue] (c\j) at ($(330:\Rbig)+({60+120*\j}:\Rsmall)$) {};
}

\node[vertex] (d) at ($(30:\Rbig)$) {};
\node[vertex] (e) at ($(150:\Rbig)$) {};
\node[vertex] (f) at ($(270:\Rbig)$) {};

\foreach \u in {a0,a1,a2} \foreach \v in {b0,b1,b2} \draw[edge] (\u) -- (\v);
\foreach \u in {a0,a1,a2} \foreach \v in {c0,c1,c2} \draw[edge] (\v) -- (\u);
\foreach \u in {b0,b1,b2} \foreach \v in {c0,c1,c2} \draw[edge] (\u) -- (\v);

\foreach \v in {b0,b1,b2} \draw[edge, color=black] (e) -- (\v);
\foreach \v in {b0,b1,b2} \draw[edge, color = black] (\v) -- (f);
\foreach \v in {c0,c1,c2} \draw[edge, color = black]  (\v) -- (d);
\foreach \v in {c0,c1,c2} \draw[edge, color = black]  (f) -- (\v);
\foreach \u in {a0,a1,a2} \draw[edge, color = black] (d) -- (\u);
\foreach \u in {a0,a1,a2} \draw[edge, color = black]  (\u) -- (e);

\draw[edge, color = black] (d) -- (e);
\draw[edge, color = black] (e) -- (f);
\draw[edge, color = black] (f) -- (d);

\foreach \x/\y in {c0/c1,c2/c0,c1/c2} \draw[edge, color = blue, line width= 1.5 pt] (\x) -- (\y);

\end{tikzpicture}
    \caption{A generalized lexicographic product $D[G_1,\ldots,G_6]$ where $|V(G_1)|=|V(G_3)|=|V(G_5)|=3$ and $|V(G_2)=|V(G_4)|=|V(G_6)|=1$.}
    \label{fig: 6 cycle mod 2}
\end{figure}
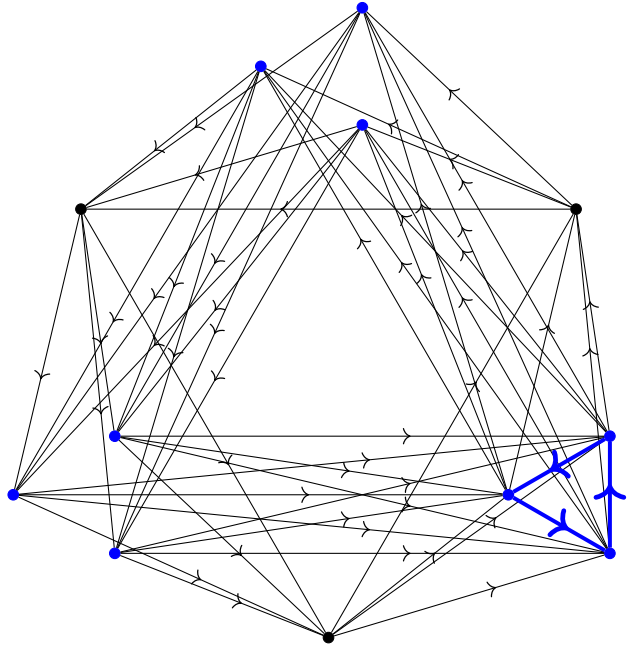

\begin{Lemma} Let $\dir{C_n}^k$ be the $k$-th power of a directed cycle, where $2k < n$. Let $d=\gcd(n,k)$. Then for all $i \in \{1,\ldots, d\}$,  the vector $\chi_i$ where
\[\chi_i(j) = \begin{cases}
    1 & \text{if } j \equiv i \mod d\\
    0 & \text{else}
\end{cases}\] lies in the kernel of $S_D$.
\end{Lemma}
\begin{proof}
    Every vertex $m \in V(\dir{C_n}^k)$ has a $+1$ in $S_D$ for all vertices in $\{m+1,\ldots,m+k\}$ and a $-1$ in $S_D$ for all vertices in $\{m-k, \ldots, m-1\}$. Since $d \mid k$ both these sets contain exactly $\frac{k}{d}$ vertices satisfying $j \equiv i \mod d$. Hence, $S_D \chi_i=0$.
\end{proof}

In light of this, every linear combination of vectors $\chi_i$ where $i \in \{1, \ldots, \frac{n}{d}\}$ lies in the kernel of $S_D$. We can use these vectors to construct more examples of Seymour-tight orientations. For example, take $D=\dir{C_6}^2$. Then $d=\gcd(6,2)=2$. Thus the vector $\chi_1+3\chi_2=(1,3,1,3,1,3)^T$ lies in the kernel of $S_D$. Hence, we can take $G_1=G_3=G_5=K_1$ to be Seymour-tight orientations on one vertex and $G_2=G_4=E_3$ and $G_6=\dir{C_3}$ to be Seymour-tight orientations on three vertices. This construction results in the graph in Figure~\ref{fig: 6 cycle mod 2}.

Lemma \ref{lem: replacement} can also be applied to other Seymour-tight orientations than lexicographic products. For example, a regular tournament can contain one or multiple small regular tournaments on which all other vertices are uniform. Then we can replace these small tournaments with some other Seymour-tight orientations to obtain a new Seymour-tight orientation that is not necessarily a regular tournament, see Figure~\ref{fig:replacement}.

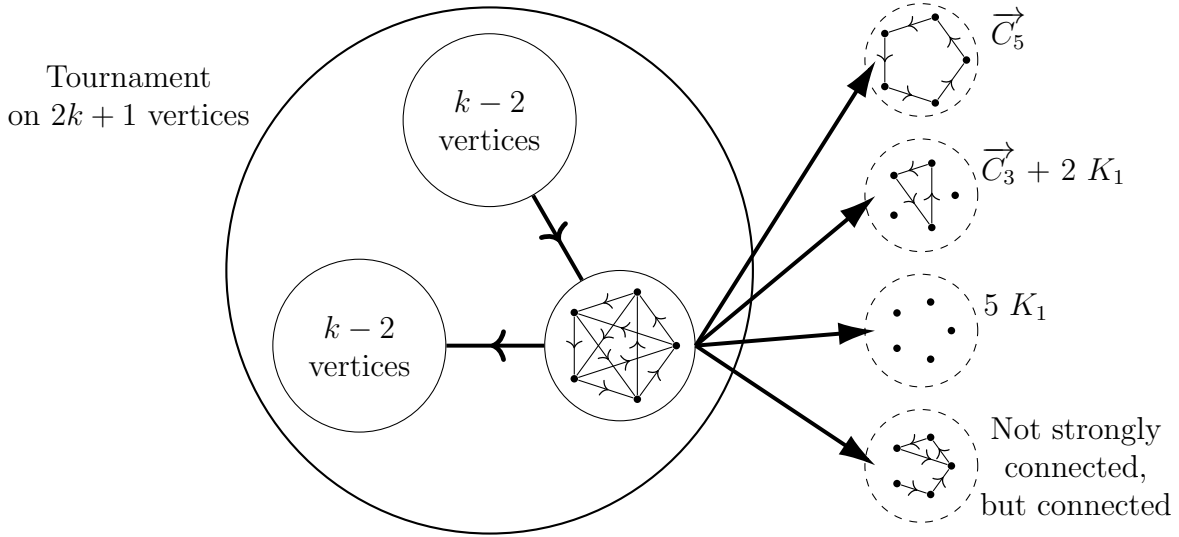
\begin{figure}[ht]
    \centering
    \begin{tikzpicture}[
    bigcycle/.style={circle, draw=black, thick, minimum size=7cm},
    smallcycle/.style={circle, draw=black, minimum size=2.3cm},
    smallercycle/.style={circle, draw=black, minimum size=2cm},
    outsidecycle/.style={circle, draw=black, dashed, minimum size=1.5cm},
    arrow/.style={-{Latex[length=5mm,width=3mm]}, thick, line width=1.5pt},
    isolated/.style={circle, draw=black, fill=black, radius=0.1cm},
    node distance=0.5cm
]

\tikzstyle{edge} = [
    postaction={
      decorate,
      decoration={
        markings,
        mark=at position 0.55 with {\arrow[scale=1.2]{>}}}}
]
\tikzstyle{edge1} = [
    postaction={
      decorate,
      decoration={
        markings,
        mark=at position 0.65 with {\arrow[scale=1.2]{>}}}}
]

\node[bigcycle] (big) {};
\node[above left, align=center] at (big.150) {Tournament\\ on $2k+1$ vertices};

\node[smallcycle, align=center] (inner1) at (210:2) {$k-2$\\ vertices};
\node[smallcycle, align=center] (inner2) at (90:2) {$k-2$\\ vertices};
\node[smallercycle, align=center] (inner3) at (330:2) {};

\node (k5) at (inner3.center) {};
\begin{scope}[shift={(inner3.center)}, scale=0.5]
    \foreach \angle in {0,72,144,216,288} {
        \node[circle, fill=black, inner sep=1pt] (v\angle) at (\angle:1.5) {};
    }
    \ \foreach \i [evaluate=\i as \angle using int(\i*72)] in {0,...,4} {
        \foreach \offset in {1,2} {
            \pgfmathtruncatemacro{\nextangle}{mod(\angle + \offset*72, 360)}
            \draw[edge] (v\angle) -- (v\nextangle);
        }
    }
\end{scope}

\draw[edge, line width = 1.5 pt] ([xshift=-1cm]inner3) -- (inner1);
\draw[edge, line width = 1.5 pt] (inner2)--([xshift=-1cm]inner3);

\coordinate (arrowstart) at ([xshift=1cm]inner3.east);
\coordinate (arrowend) at ([xshift=3cm,yshift=1cm]inner3.east);

\draw[arrow] (inner3.east) -- ([yshift=-0.8cm, xshift=-0.6cm]arrowend);
\draw[arrow] (inner3.east) -- ([yshift=2.8cm,xshift=-0.6cm]arrowend);
\draw[arrow] (inner3.east) -- ([yshift=-2.6cm,xshift=-0.6cm]arrowend);
\draw[arrow] (inner3.east) -- ([yshift=1.0cm,xshift=-0.6cm]arrowend);

\node[outsidecycle] (outside1) at ([yshift=2.8 cm]arrowend) {};
\node[outsidecycle] (outside2) at ([yshift=1.0 cm]arrowend){};
\node[outsidecycle] (outside3) at ([yshift=-0.8cm]arrowend) {};
\node[outsidecycle] (outside4) at ([yshift=-2.6cm]arrowend) {};

\node[above] at ([xshift=0.4cm]outside1.east) {$\overrightarrow{C_5}$};
\node[above] at ([xshift=1cm]outside2.east) {$\overrightarrow{C_3}$ + 2 $K_1$};
\node[above] at ([xshift=0.5cm]outside3.east) {5 $K_1$};
\node[align=center] at ([xshift=1.3 cm]outside4.east) {Not strongly\\ connected,\\ but connected};

\begin{scope}[shift={(outside1.center)}, scale=0.3]
    \foreach \angle in {0,72,144,216,288} {
        \node[circle, fill=black, inner sep=1pt] (c5\angle) at (\angle:2) {};
    }
    \foreach \angle [count=\i] in {0,72,144,216,288} {
        \pgfmathtruncatemacro{\nextangle}{mod(\angle+72,360)}
        \draw[edge] (c5\angle) -- (c5\nextangle);
    }
\end{scope}

\begin{scope}[shift={(outside2.center)}, scale=0.3]
    \foreach \angle in {0, 72,144,216,288} {
        \node[circle, fill=black, inner sep=1pt] (c3\angle) at (\angle:1.5) {};
    }
    \draw[edge] (c3288) -- (c372);
    \draw[edge1] (c372) -- (c3144);
    \draw[edge] (c3144) -- (c3288);
\end{scope}

\begin{scope}[shift={(outside3.center)}, scale=0.2]
    \foreach \angle in {0,72,144,216,288} {
        \node[circle, fill=black, inner sep=1pt] (iso\angle) at (\angle:2) {};
    }
\end{scope}

\begin{scope}[shift={(outside4.center)}, scale=0.2]
    \foreach \angle in {0,72,144,216,288} {
        \node[circle, fill=black, inner sep=1pt] (d\angle) at (\angle:2) {};
        }
    \draw[edge1] (d0) -- (d72);
    \draw[edge1] (d72) -- (d144);
    \draw[edge1] (d144) -- (d0);
    \draw[edge1] (d216) -- (d288);
    \draw[edge1] (d288) -- (d0);

\end{scope}
\end{tikzpicture}
    \caption{If a large regular tournament  on $2k+1$ vertices contains a small regular tournament such that all other vertices are uniform on that tournament, then we can replace the small regular tournament with some other Seymour-tight orientation.}
    \label{fig:replacement}
\end{figure}

\begin{Example}\label{example}
    Let $k$ and $l$ be integers such that $3l+1 \leq k$. Let $T$ be a regular tournament on $2k+1$ vertices such that there exists a set $X$ of size $2l+1$ that all vertices in $X$ have the same in- and out- neighbourhood. Since $3l+1 \leq k$, this is possible as there are $2k+1-(2l+1) \geq 4l+2$ vertices outside $X$ in $T$. 
    
    Since $X$ is an induced subgraph of a tournament, $X$ itself is also a tournament. Moreover, there exists $k$ such that 
    $k = |\NGi{1}{T|_X}{v}|-|\NGi{2}{T|_X}{v}|$ for all $v \in X$.  Fischer's theorem~\cite{fisher1996squaring} implies $k \geq 0$. Moreover, there is at least one $w \in X$ that has out-degree at least $\frac{X-1}{2}$ implying  $k \leq 0$. Hence, $k=0$. Thus replacing $X$ with any Seymour-tight orientation on $|X|$ vertices in $G$ results in a new Seymour-tight orientation.
\end{Example}


\subsection{Seymour-tight orientations with out-degree at most 2}
Throughout the paper, we will see that lexicographic products are a powerful tool; in this section, we see that it is the only tool we need to characterize strongly connected Seymour-tight orientations with out-degree at most 2.
Note that a graph $D$ containing a vertex $v$ of out-degree zero is strongly connected only if $D=\{v\}$. Hence, we may assume that every vertex in $D$ has at least one out-neighbour. We will give a characterization based on whether there exists a vertex of out-degree exactly $1$ or $2$. 

Directed cycles are strongly connected Seymour-tight orientations in which every vertex has out-degree one. We will now show these are the only strongly connected Seymour-tight orientation that have a vertex of out-degree one. 

\begin{Lemma}\label{lem:out-degree1}
 A strongly connected orientation $D$ is a Seymour-tight orientation with a vertex of out-degree 1 if and only if $D$ is a directed cycle.
\end{Lemma}
\begin{proof} First, note that every directed cycle is a strongly connected Seymour-tight orientation.

Let $D$ be a strongly connected Seymour-tight orientation and suppose that $v$ has one out-neighbour $w$. Then \[|\NGi{1}{D}{w}|=|\NGi{2}{D}{v}|=|\NGi{1}{D}{v}|=1,\] and also $w$ has exactly one out-neighbour. Repeating this argument and using the fact that $D$ is strongly connected, we conclude that every vertex in $D$ has exactly one out-neighbour. Therefore, $D$ must be a directed cycle. 
\end{proof}

Next, we characterize all strongly connected Seymour-tight orientations with a vertex of out-degree 2.

\begin{Lemma}\label{lem:out-degree2}
    An strongly connected orientation $D$ is a Seymour-tight orientation with a vertex of out-degree 2 if and only if $D$ is isomorphic to one of  $\dir{C_n}^2$ or $\dir{C_n}[E_2]$ for some $n$.
\end{Lemma}
\begin{proof}
    Let $D$ be a strongly connected Seymour-tight orientation and suppose that $v_0$ has two out-neighbours $v_1$ and $v_2$. Lemma~\ref{lem:out-degree1} implies that every vertex has at least two out-neighbours, since the directed cycle has no vertex of out-degree $2$.
    
    First, consider the case where there is an arc from $v_1$ to $v_2$. Then
    \[|\NGi{1}{D}{v_2}| \leq |\NGi{2}{D}{v_0}|=|\NGi{1}{D}{v_0}|=2,\]
    and thus $v_2$ has at most two out-neighbours, say $\NGi{2}{D}{v_0} = \{v_3,v_4\} $. Then $\NGi{1}{D}{v_1}\subseteq \{v_2,v_3,v_4\}$. If $\NGi{1}{D}{v_1}= \{v_2,v_3,v_4\}$, then \begin{align*}
    \NGi{2}{D}{v_1} &= (\NGi{1}{D}{v_2} \cup \NGi{1}{D}{v_3} \cup \NGi{1}{D}{v_4})-\NGi{1}{D}{v_1}\\ &\subseteq  \NGi{1}{D}{v_3} \cup \NGi{1}{D}{v_4} = \NGi{2}{D}{v_2}.\end{align*} 
    Thus, $|\NGi{2}{D}{v_1}| \leq |\NGi{2}{D}{v_2}| =2$, a contradiction. Since the out-degree of any vertex is at least $2$, we can assume without loss of generality that $\NGi{1}{D}{v_1}=\{v_2,v_3\}$. Since there is an arc from $v_2$ to $v_3$, we may  repeat the argument above and conclude that we can number the vertices of $D$ from $v_0,\ldots, v_n$ such that there is an arc from every $v_i$ to $v_{i+1}$ and $v_{i+2}$. Since we can also start our procedure with $v_1$ instead of $v_0$, we remark that $v_0$ should be equal to $v_{n+1}$. This is exactly the definition of $\dir{C_n}^2$.

    Suppose now that there is no arc from $v_1$ to $v_2$. By the argument above, $D$ has no vertex $v$ of out-degree $2$ for which there is an arc between the two neighbours of $v$. Note that
    \[|\NGi{1}{D}{v_1} \cup \NGi{1}{D}{v_2}| = |\NGi{2}{D}{v_0}|=|\NGi{1}{D}{v_0}|=2.\]
    As both $v_1$ and $v_2$ have at least two out-neighbours, we see that \[\NGi{1}{D}{v_1}=\NGi{1}{D}{v_2}=\{v_3,v_4\},\] where there is no arc between $v_3$ and $v_4$. Hence, we can repeat the argument, but now starting with $v_1/v_2$ instead of $v_0$. Thus, by induction, we can pair the vertices of $D$ into pairs of the form $v_{2k-1},v_{2k}$ such that there are arcs from each of $v_{2k-1},v_{2k}$ to both $v_{2k+1}$ and $v_{2k+2}$ for every $k$. This is precisely the definition of $\dir{C_n}[E_2]$.

    For the converse direction, $\dir{C_n}^2$ is a Seymour-tight orientation by Lemma~\ref{lem: kth power dicycle} and $\dir{C_n}[E_2]$ is a Seymour-tight orientations by Lemma~\ref{Seymour lex}.
\end{proof}

If a strongly connected Seymour-tight orientation $D$  has a vertex of out-degree $1$ or $2$, then it follows from Lemmas~\ref{lem:out-degree1} and~\ref{lem:out-degree2} that $D$ is out-regular. We note that this property does not extend to higher out-degree;
as seen in Figure~\ref{fig:non_reg_example}, there are Seymour-tight orientations that have a vertex of out-degree 3 that also contain vertices with higher out-degrees.

\section{Strongly disconnected Seymour-tight orientations}\label{Sec: Strongly disconnected}

As we have seen in Section~\ref{Sec: examples}, every strongly disconnected Seymour-tight orientation can be formed by adding source components to a Seymour-tight orientation. In this section, we give two such constructions. 

\begin{Lemma}\label{Source copy neighbourhood}
    Let $D$ and $G$ be two Seymour-tight orientations. Suppose $|\NGi{1}{G}{X}| = |X|$ for $X \subseteq V(G)$. Let $H$ be the orientation $E(D) \cup E(G)$ together with all arcs from $u \rightarrow x$ where $u \in D$ and $x \in X$. Then $H$ is a Seymour-tight orientation.
    
    In particular, we can take $X=\NGi{1}{G}{x}$ for any $x \in V(G)$.
\end{Lemma}
\begin{proof}
     By the definition of $H$, the first and second neighbourhoods of all vertices in $V(G)$ are the same as in $G$, thus their first and second neighbourhood in $H$ have the same size. Let $v \in V(D)$ be arbitrary. Then \[|\NGi{1}{H}{v}| = |\NGi{1}{D}{v}|+|X|.\]
    A vertex that can be reached within two steps from $v$ in $D$, is either a neighbour of a vertex in $X$ or a neighbour of a vertex in $V(D)$. The neighbours of $X$ are the vertices in $\NGi{1}{G}{X}$ and $\NGi{1}{H}{w} \subseteq V(D) \cup X$ for all $w \in V(D)$. Hence,
    \[|\NGi{2}{H}{v}| = |\NGi{2}{D}{v}|+|N(X)| = |\NGi{1}{D}{v}|+|X|=|\NGi{1}{H}{v}|.\]
    for all $v \in V(D)$, and thus $H$ is a Seymour-tight orientation. 

    If  $X = \NGi{1}{G}{x}$ for any $x \in V(G)$, then \[|\NGi{1}{G}{X}|=|\NGi{1}{G}{\NGi{1}{G}{x}}| = |\NGi{2}{G}{x}| = |\NGi{1}{G}{x}| = |X|. \qedhere\]  
\end{proof}

\begin{Lemma}\label{Graph hom construction}
    Let $D$ and $G$ be two Seymour-tight orientations. Let $f: D \rightarrow G$ be a digraph homomorphism. 
    \begin{enumerate}
        \item  Let $H$ be the orientation on $V(D)+V(G)$ with arcs $A(D)$, $A(G)$ and all arcs $d \rightarrow g$ with $d \in D$, $g \in V(G)$ such that $g \in \NGi{1}{G}{f(d)}$. Then $H$ is a Seymour-tight orientation.
        \item Suppose that $f$ is a bijective graph homomorphism, thus $|V(D)|=|V(G)|$. Let $O$ be the orientation on $V(D)+V(G)$ with arcs $A(D)$, $A(G)$ and all arcs $g \rightarrow d$ with $g \in V(G), d \in V(D)$ satisfying $f(d) \in \NGi{1}{G}{g}$. Then $O$ is a Seymour-tight orientation.
    \end{enumerate}
\end{Lemma}
\begin{proof}
We first prove that $H$ is a Seymour-tight orientation.
    Note that there are no arcs from $G$ to $ D$ in $H$. Since $G$ is a Seymour-tight orientation, $|\NGi{1}{H}{g}| = \NGi{2}{H}{g}$ for all $g \in V(G)$. For any $d \in V(D)$, we have
\[|\NGi{1}{H}{d}|=|\NGi{1}{D}{d}|+|\NGi{1}{G}{f(d)}|. \]

Let $w \in \NGi{2}{H}{d} \cap G$. Then $w \not\in \NGi{1}{H}{d}$, thus $w \not \in \NGi{1}{G}{f(d)}$. Moreover, there exists a vertex $v$ such that $d \rightarrow v \rightarrow w$.
We extend $f$ to  a graph homomorphism $D \cup G \rightarrow G$ such that $f(x)=x$ for all $x \in V(G)$. Hence, $f(d) \rightarrow f(v) \rightarrow f(w)=w$ in $G$. Since $w \not \in \NGi{1}{G}{f(d)}$, we obtain $w \in \NGi{2}{G}{f(d)}$. As every vertex of  $\NGi{2}{G}{f(d)}$ is a second neighbour of $d$ in $H$, we obtain $ \NGi{2}{H}{d} \cap V(G)=\NGi{2}{G}{f(d)}$. Since there are no arcs from $G$ to $D$, we obtain for all $d \in D$
\[|\NGi{2}{H}{d}| =|\NGi{2}{D}{d}|+|\NGi{2}{G}{f(d)}|. \]
Since $D$ and $G$ are Seymour-tight orientations, we have for all $d \in D$
\[|\NGi{2}{H}{d}| =|\NGi{2}{D}{d}|+|\NGi{2}{G}{f(d)}|=|\NGi{1}{D}{d}|+|\NGi{1}{G}{f(d)}|=|\NGi{1}{H}{d}|.\]

We now prove that $O$ is a Seymour-tight orientation.
    Note that there are no arcs from $D$ to $G$ in $O$. Since $D$ is a Seymour-tight orientation, $\NGi{1}{O}{d} = \NGi{2}{O}{d}$ for all $d \in V(D)$. Since $f$ is an bijection, it has a bijective inverse $f^{-1}: G \rightarrow D$, which is a graph cohomomorphism (but not necessarily a graph homomorphism). 
    For any $g \in G$, we have $|\NGi{1}{O}{g}|=2|\NGi{1}{G}{g}|$ since
    $g' \in \NGi{1}{G}{g}$ if and only if $g', f^{-1}(g') \in \NGi{1}{O}{g}$. Let $d \in \NGi{2}{O}{g} \cap D$, then $d \not \in \NGi{1}{O}{g}$, thus $f(d) \not \in \NGi{1}{G}{g}$. Moreover, there exists $d'$ such that $g \rightarrow d' \rightarrow d$ in $H$. We extend $f$ to a homomorphism $D \cup G \rightarrow G$ such that $f(x)=x$ for all $x \in V(G)$. Then $g=f(g) \rightarrow f(d') \rightarrow f(d)$. Since $f(d) \not \in \NGi{1}{G}{g} $, we obtain $f(d) \in \NGi{2}{G}{g}$. Since there are no arcs from $D$ to $G$, we obtain for all $g \in G$ 
    \[ |\NGi{2}{O}{g}|=2|\NGi{2}{G}{g}|=2|\NGi{1}{G}{g}| =|\NGi{1}{O}{g}|. \qedhere\]
\end{proof}

\section{Sullivan's conjecture} \label{Section: Sullivans conjecture}

Let $G$ be a directed graph. Then $\NGi[-]{1}{G}{v} = \{u \in V(G) \mid uv \in E(G)\}$ is the \textsl{in-neighbourhood} of the vertex $v$. In her survey on the Caccetta-H\"aggkvist conjecture~\cite{sullivan2006summary}, Sullivan proposed the following variation of Seymour’s conjecture.

\begin{Con}[Sullivan~\cite{sullivan2006summary}]
    Every oriented graph contains at least one vertex such that  $|N^+_2(v)| \geq |N^{-}_1(v)|$. 
\end{Con}

Note that this conjecture coincides with Seymour’s second neighbourhood conjecture when restricted to Eulerian orientations. Sullivan’s conjecture has received significantly less attention. Nevertheless, it is known to hold for tournaments, for graphs in which the number of transitive triangles is small relative to the number of arcs, and for almost all oriented graphs~\cite{ai2024seymour}. Analogous to Seymour-tight orientations, we call an orientation $G$ a \textsl{Sullivan-tight orientation} if $|\NGi[-]{1}{G}{v}|=|\NGi[+]{2}{G}{v}|$ for all $v \in V(G)$.

We will start by giving a few examples. 

\begin{Lemma}
     Let $\dir{C_n}$ be a directed cycle. Let $k$ be a natural number such that $2k < n$. Then the $k$-th power of $\dir{C_n}$, denoted by $\dir{C_n}$, is a Sullivan-tight orientation.
\end{Lemma}
\begin{proof}
    Let $v_i$ be a vertex in the $k$-th power $\dir{C_n}^k$. Then $\NGi{1}{\dir{C_n}^k}{v_i} = \{v_{i+1}, \ldots, v_{i+k}\}$. Therefore, the vertices that can be reached in at most two steps from $v_i$ are the vertices $v_{i+1}, \ldots, v_{i+2k}$. Therefore, $\NGi{2}{\dir{C_n}^k}{v_i} = \{v_{i+k+1}, \ldots, v_{i+2k}\}$, which implies $|\NGi{1}{\dir{C_n}^k}{v_i}| = k = |\NGi{2}{\dir{C_n}^k}{v_i}|$ for all vertices $v_i$.
\end{proof}

\begin{Lemma}
    A tournament is a Sullivan-tight orientation if and only if it  has diameter 2.
\end{Lemma} 
\begin{proof}
    A tournament $T$ has diameter $2$ if and only if for all $v \in V(G)$, we have that $V(T)=v \cup \NGi[+]{1}{T}{v} \cup  \NGi[+]{2}{T}{v}$. Hence, $\NGi[-]{1}{T}{v} = V(T)-\{v\}- \NGi[+]{1}{T}{v} = \NGi[+]{2}{T}{v}$ for all $v \in V(T)$ if and only if the diameter of $T$ is 2.
\end{proof}

Note that every regular tournament has diameter 2, but not every tournament with diameter 2 is a regular tournament. Hence, there are Sullivan-tight orientations that are not Seymour-tight orientations, see Figure~\ref{fig: Sullivan not Seymour}. We now prove that Sullivan-tight orientations do not have sinks and all sources should be universal sources. Hence, our constructions for Seymour-tight orientations from Section~\ref{Sec: Strongly disconnected} do not extend for Sullivan-tight orientations. 

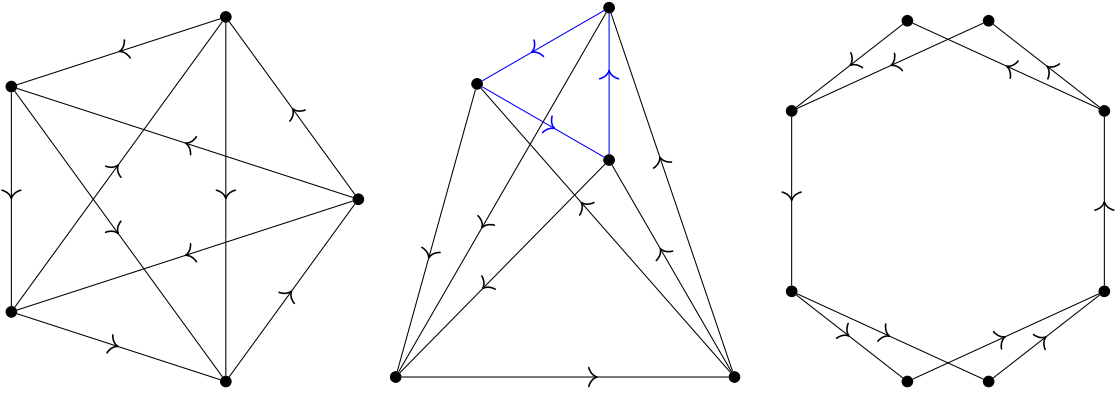
\begin{figure}[ht]
\begin{subfigure}[t]{0.3 \linewidth}
     \centering
    \begin{tikzpicture}[scale=0.85]
\tikzstyle{edge} = [
    postaction={
      decorate,
      decoration={
        markings,
        mark=at position 0.5 with {\arrow[scale=1.5]{>}}}}
]
\tikzstyle{vertex}=[circle, draw, fill=black,
                        inner sep=0pt, minimum width=4pt]

\def\Rbig{3}    

\foreach \j in {0,1,2,3,4} {
  \node[vertex] (a\j) at ($(72*\j:\Rbig)$) {};
}

\foreach \j in {0,1,2,3}{
\draw[edge] (a\j)--(a\number\numexpr\j+1\relax);
}
    
\draw[edge] (a4)--(a0);
\draw[edge] (a0)--(a2);
\draw[edge] (a2)--(a4);
\draw[edge] (a0)--(a3);
\draw[edge] (a3)--(a1);
\draw[edge] (a1)--(a4);

\end{tikzpicture}
\end{subfigure}
 \begin{subfigure}[t]{0.3 \linewidth}
     \centering
    \begin{tikzpicture}[scale=1.3]
\tikzstyle{edge} = [
    postaction={
      decorate,
      decoration={
        markings,
        mark=at position 0.6 with {\arrow[scale=1.5]{>}}}}
]
\tikzstyle{vertex}=[circle, draw, fill=black,
                        inner sep=0pt, minimum width=4pt]

\def\Rbig{2}    
\def\Rsmall{0.9}  

\foreach \j in {0,1,2} {
  \node[vertex] (a\j) at ($(90:\Rbig)+({60+120*\j}:\Rsmall)$) {};
}
\foreach \j in {0} {
  \node[vertex] (b\j) at ($(210:\Rbig)$) {};
}
\foreach \j in {0} {
  \node[vertex] (c\j) at ($(330:\Rbig)$) {};
}

\foreach \u in {a0,a1,a2} \foreach \v in {b0} \draw[edge] (\u) -- (\v);
\foreach \u in {a0,a1,a2} \foreach \v in {c0} \draw[edge] (\v) -- (\u);

\draw[edge] (b0)--(c0);
\draw[edge, color=blue] (a0)--(a1);
\draw[edge, color=blue] (a1)--(a2);
\draw[edge, color=blue] (a2)--(a0);

\end{tikzpicture}
\end{subfigure}
\begin{subfigure}[t]{0.3 \linewidth}
     \centering
    \begin{tikzpicture}[scale=0.6]
\tikzstyle{edge} = [
    postaction={
      decorate,
      decoration={
        markings,
        mark=at position 0.5 with {\arrow[scale=1.5]{>}}}}
]
\tikzstyle{vertex}=[circle, draw, fill=black,
                        inner sep=0pt, minimum width=4pt]

\def\Rbig{4}    
\def\Rsmall{0.9}  

\foreach \j in {0,1} {
  \node[vertex] (a\j) at ($(90:\Rbig)+({180*\j}:\Rsmall)$) {};
}
\foreach \j in {0,1} {
  \node[vertex] (b\j) at ($(270:\Rbig)+({180*\j}:\Rsmall)$) {};
}

\node[vertex] (d) at ($(30:\Rbig)$) {};
\node[vertex] (e) at ($(150:\Rbig)$) {};
\node[vertex] (f) at ($(210:\Rbig)$) {};
\node[vertex] (g) at ($(-30:\Rbig)$) {};

\draw[edge] (g)--(d);
\draw[edge] (d)--(a1);
\draw[edge] (a1)--(e);
\draw[edge] (d)--(a0);
\draw[edge] (a0)--(e);
\draw[edge] (e)--(f);
\draw[edge] (f)--(b1);
\draw[edge] (b1)--(g);
\draw[edge] (f)--(b0);
\draw[edge] (b0)--(g);


\end{tikzpicture}
\end{subfigure}
\caption{Three Sullivan-tight orientations that are not Seymour-tight.}\label{fig: Sullivan not Seymour}
\end{figure}

\begin{Lemma}
    No Sullivan-tight orientation has a sink. Moreover, every source in a Sullivan-tight orientation is connected to a set $X$ satisfying $N^{+}_1(X)=\emptyset$.
\end{Lemma}
\begin{proof}
    A sink $v$ in a Sullivan-tight orientation $G$ satisfies $\emptyset = \NGi[+]{1}{G}{v}$, implying $\NGi[+]{2}{G}{v}=\emptyset$. Hence, $|\NGi[-]{1}{G}{v}|=|\NGi[+]{1}{G}{v}|=0$ and thus has $v$ is an isolated vertex and not a sink.

    A source $w$ in a Sullivan-tight orientation $G$ satisfies $\NGi[-]{1}{G}{w}=\emptyset$, hence $|\NGi[+]{2}{G}{w}|= |\NGi[-]{1}{G}{w}|=0$. If $X=\NGi[+]{1}{G}{w}$, then $N^{+}_1(X)=\emptyset.$
\end{proof}

However, (generalized) lexicographic products preserve not only Seymour-tight orientations, but also Sullivan-tight orientations.

\begin{Lemma}\label{Sullivan lex}
Let $D$ and $G$ be two Sullivan-tight orientations. Then the lexicographic product $D[G]$ is also a Sullivan-tight orientation.
\end{Lemma}
\begin{proof}
Let $D$ and $G$ be two Seymour-tight orientations. Let $(v,i)$ be any vertex in $D[G]$. Then
    \[\NGi[-]{1}{D[G]}{(v,i)} = \left\{(w,j) \, \middle\vert \, w \in \NGi[-]{1}{D}{v} \textbf{ or } v=w \text{ and } j \in \NGi[-]{1}{G}{i}\right\}.\]
    In particular, $|\NGi[-]{1}{D[G]}{(v,i)}| = |V(G)| \cdot |\NGi[-]{1}{D}{v}| + |\NGi[-]{1}{G}{i}|$. Moreover,
    \[\NGi[+]{1}{D[G]}{(v,i)} = \left\{(w,j) \, \middle\vert \, w \in \NGi[+]{1}{D}{v} \textbf{ or } v=w \text{ and } j \in \NGi[+]{1}{G}{i}\right\}.\]

    All vertices $(w,j)$ that can be reached from $(v,i)$ in at most two steps satisfy either $w \in \NGi[+]{1}{D}{v} \, \cup \, \NGi[+]{2}{D}{v}$ or $w=v$ and $j \in \NGi[+]{1}{G}{i} \, \cup \, \NGi[+]{2}{G}{i}$. By deleting those in $\NGi[+]{1}{D[G]}{(v,i)}$, we obtain
    \[\NGi[+]{2}{D[G]}{(v,i)} = \left\{(w,j) \, \middle\vert \, w \in \NGi[+]{2}{D}{v} \textbf{ or } v=w \text{ and } j \in \NGi[+]{2}{G}{i}\right\}.\]
    This implies that $\NGi[+]{2}{D[G]}{(v,i)} = |V(G)| \cdot |\NGi[+]{2}{D}{v}| + |\NGi[+]{2}{G}{i}|$. Since $D$ and $G$ are Sullivan-tight orientations, we have
    \begin{align*} |\NGi[-]{1}{D[G]}{(v,i)}| &= |V(G)| \cdot |\NGi[-]{1}{D}{v}| + |\NGi[-]{1}{G}{i}|\\&= |V(G)| \cdot |\NGi[+]{2}{D}{v}| + |\NGi[+]{2}{G}{i}|=  |\NGi[+]{2}{D[G]}{(v,i)}| \end{align*}
    for all vertices $(v,i)$. Hence, $D[G]$ is also a Sullivan-tight orientation.
\end{proof}

Notice in the proof that the only information about $G$ we needed was $|V(G)|$ and by the same argument have the following result.
\begin{Lemma}
 Let $D$ be a Sullivan-tight orientation on $n$ vertices. Let $G_1,\ldots,G_n$ be Sullivan-tight orientations on $k$ vertices. Then the generalized lexicographic product $D[G_1,\ldots,G_n]$ is a Sullivan-tight orientation.
\end{Lemma}

Let $D$ be an orientation. Let $R_D$ be the matrix defined by $R_D(v,w)=1$ if $w \in \NGi[+]{2}{D}{v}\setminus \NGi[-]{1}{D}{v}$ and $R_D(v,w)=-1$ if  $w \in \NGi[-]{1}{D}{v} \setminus \NGi[+]{2}{D}{v}$ and $R_D(v,w)=0$ elsewhere. In particular, $R_D(v,w)=0$ if  $w \in \NGi[-]{1}{D}{v} \cap \NGi[+]{2}{D}{v}$.  By definition, $S_D \mathbf{1}=0$ if and only if $D$ is a Sullivan-tight orientation. 

\begin{Lemma}
    Let $D$ be an orientation on $[n]$ and let $\mathbf{x} \in \mathbb{Z}_{\geq 0}$ be a vector such that $R_D\mathbf{x} = 0$. For all $i \in [n]$, let $G_i$ be a Sullivan-tight orientation of size $\mathbf{x}_i$. Then the graph $D[(G_i)_{i \in V(D)}]$ is a Sullivan-tight orientation.
\end{Lemma}

Since $R_D(v,w)=0$ if  $w \in \NGi[-]{1}{D}{v} \cap \NGi[+]{2}{D}{v}$, there are graphs for which there are many zeros in $R_D$ implying a large kernel. For example, for a directed triangle $\dir{C_3}$, we have $R_{\dir{C_3}}=\textbf{0}$. Hence, any vector in $\mathbb{Z}_{>0}$ lies in the kernel of $R_{\dir{C_3}}$.

\begin{Cor}
    Let $G_1, G_2,G_2$ be three Sullivan-tight orientations, then $\dir{C_3}[G_1,G_2,G_3]$ is a Sullivan-tight orientation.
\end{Cor}

\subsection{Putative counterexamples to Sullivan's conjecture}

We can also prove analogous results to those in Subsection~\ref{Section: Putative counterexamples}. As the proofs are very similar to the proofs of their respective results for Seymour-tight orientations, for brevity we only give proof sketches. 

\begin{Theorem} Let $G$ be a counterexample to Sullivan's conjecture. Suppose that $H$ satisfies $|\NGi[-]{1}{H}{v}| \geq |\NGi[+]{2}{H}{v}|$ for all $v \in v(H)$. Then $G[H]$ and $H[G]$ are counterexamples.
\end{Theorem}

\begin{Cor}
    If Sullivan's conjecture is false, then there exists $\epsilon > 0$ for which there exists a strongly connected orientation $O$ with arbitrary high minimum out-degree such that no vertex $v$ in $O$ satisfies $|N^+_2(v)| \geq (1-\epsilon) |N^{-}_1(v)|$.
\end{Cor}
\begin{proof}
    Let $H$ be a counterexample, then we look at the sequence $H[H[\ldots[H]\ldots]]$. We can now use the same arguments as in the proof of Lemma~\ref{lem: sequence of counterexamples}.
\end{proof}

\begin{Cor}
    If Sullivan's conjecture is false, then there exists $k \in \mathbb{N}$ for which there are strongly connected counterexamples with arbitrarily many vertices such that $\Delta^{+} \leq k$. Moreover, there are also counterexamples for which the orientation satisfies $\delta^{+} \geq \frac{n}{2}-k$.
\end{Cor}
\begin{proof}
    Take the lexicographic product $G[H]$ where $H$ is a counterexample to Sullivan's conjecture. Take $G$ to be $\overline{C_n}$ where $n \gg |V(H)|$. To construct an example with high out-degree, take $G$ to be a regular tournament $T_n$ where again $n \gg |V(H)|$.
\end{proof}

\section{Seymour Cayley orientations}\label{Section: Cayley Seymour}

In this section, we will look at Cayley digraphs that are also Seymour orientations. Recal that any Cayley graph is vertex transitive and therefore any Cayley Seymour orientation is either a counterexample or Seymour-tight.
Let $\Gamma=\Gamma(G,S)$ be a Cayley digraph and a Seymour orientation. 
The connection set $S$ satisfies $S \cap S^{-1} = \emptyset$. Moreover, $|\NGi[+]{1}{G}{v}| = |S|$ and $|\NGi[+]{1}{G}{v} \cup \NGi[+]{2}{G}{v}|=|S^2 \cup S|$, thus it must satisfy $|S \cup S^2| \leq 2|S|$.  Let $S_1=S\cup\{1\}$ where $1$ is the identity element. Then $S_1^2= S \cup S^2 \cup \{1\}$ and since, $S \cap S^{-1} = \emptyset$, the  only way to write $1$ as the product of two elements in $S_1$ is as $1=1\cdot 1$. In particular,
\[|S_1^2| = |S \cup S^2| + 1 \leq 2|S|+1 = 2|S_1|-1.\]
A pair $A,B \subseteq G$ for which $|AB| < |A|+|B|$ is called a \textsl{critical pair}. The study of critical pairs is an important topic within structural additive combinatorics~\cite{Grynkiewicz2013}. Kemperman proved the following result. 

\begin{Lemma}[Kemperman~\cite{kemperman1956complexes, olson1984sum}]\label{Lem: Kemperman}
Let $G$ be a group and let $A$ and $B$ be a finite subset of $G$ with $1 \in A \cap B$. If $1 = ab$ with $a \in A$ and $b \in B$ implies $a=b=1$, then
\[|AB| \geq |A|+|B|-1.\]
\end{Lemma} 

 Kemperman's Lemma implies $|S_1^2| = 2|S_1|-1$ and thus $S$ satisfies $|S^2 \cup S| = 2 |S|$. Therefore, any Seymour Cayley orientation is Seymour-tight. In particular, there is no Cayley counterexample, which was first observed by Hamidoune~\cite{hamidoune1981application}. From now on, we will write Seymour orientation for brevity (although they are actually Seymour-tight).

To classify all Seymour Cayley orientations, we have to find all sets $S$ for which the inequality in Lemma~\ref{Lem: Kemperman} is tight for the pair $(S,S)$. Kemperman \cite{kemperman1960small} classified all such pairs of sets in abelian groups. We will use this result to prove that any Seymour Cayley orientation is the (possibly repeated) lexicographic products of empty graphs, the $k$-th power of directed cycles, and regular tournaments. We first need to introduce some notation; note we will use additive notation for groups while discussing abelian groups.

\begin{Def}[\cite{kemperman1960small}]
A pair $(A',B')$ of non-empty finite subsets in an abelian group $G$ is an
\textsl{elementary critical pair} if at least one of the following conditions holds:
\begin{enumerate}
\item[(a)] Either $|A'|=1$ or $|B'|=1$.

\item[(b)] $A'$ and $B'$ are arithmetic progressions with a common difference
$d$ such that the order of $d$ satisfies $d\ge |A'|+|B'|-1$. Thus $A'+B'$ is an arithmetic
progression of difference $d$, and there is at least one element that can be uniquely written as the sum of an element in $A'$ and one in $B'$.

\item[(c)] For some finite subgroup $H$, both $A'$ and $B'$ are contained in an
$H$-coset and
$|A'|+|B'|=|H|+1$.
Here $A'+B'$ is an $H$-coset and there is precisely one element that can be uniquely written as the sum of an element in $A'$ and in $B'$.

\item[(d)] $A'$ is aperiodic and for some finite subgroup $H$ of $G$,
$A'$ is contained in an $H$-coset. $B'$ is of the form $B'=g_0-(\overline{A'} \cap  (a+H))$ where $a\in A'$. Hence, $A'+B' = (g_0+H)-g_0$. In this case, no element can be uniquely written as the sum of an element in $A'$ and in $B'$.
\end{enumerate}
\end{Def}

Observe that each of the conditions $(a)-(d)$ implies $|A'+B'|=|A'|+|B'|-1$. We reformulate \cite[Theorem 5.1]{kemperman1960small} below, avoiding specialized notation.

\begin{Theorem}\cite[Theorem 5.1]{kemperman1960small} \label{Thm equality Kemperman}
Let $G$ be an abelian group with $|G|\geq 2$. Let $A,B$ be finite non-empty
subsets of $G$ such that 
\[
|A+B|=|A|+|B|-1 
\]
and suppose there exists an element $c \in A+B$ having a \textsl{unique representation} $c=a_0+b_0$ with $a_0 \in A$ and $b_0 \in B$. 
Then there exists non-empty subsets $A'\subseteq A$ and $B'\subseteq B$, and a subgroup $F\leq G$ of order $|F| \geq 2$, together with a quotient map $\phi: G \rightarrow G \slash F$, such that all of the following hold.
\begin{enumerate}
\item[(i)] The pair $(A',B')$ is an elementary critical pair, and each of $A',B'$ is contained
in an $F$-coset.

\item[(ii)] The element $\delta=\phi(A'+B')$ in $G \slash F$
has $\delta=\phi( A')+\phi( B')$ as its only representation of the form
$\delta=\phi(a)+\phi (b)$, with $a\in A$ and $b\in B$. 

\item[(iii)] The complement $A\backslash A'$ satisfies $(A\backslash A')+F=(A\backslash A')$, and similarly,
$(B\backslash B')+F = (B\backslash B')$. Hence, from (ii), the
complement $C'$ of $A'+B'$ in $A+B$ satisfies $C'+F=C'$.

\item[(iv)] Finally, $|\phi (A + B)|=|\phi (A)|+|\phi (B)|-1$.
\end{enumerate}
\end{Theorem}

Based on this theorem, we can classify all Seymour orientations that are  Cayley digraphs of an abelian group.

\ClassificationAbelian*
\begin{proof}
    Recall that we only need to classify the Seymour-tight orientations. We will prove this theorem with induction on the number of vertices. If the graph has one vertex, it is clearly of the correct form.
    
    Let $S$ be the connection set of a Seymour-tight orientation in an abelian group $G$. Let $A = S \cup \{0\}$. Then $|A+A| = 2 |S|+|\{0\}|= 2|A|-1$. Moreover, if $0 = a+a'$ where $a,a'\in A$, then $a=a'=0$. Hence, $(A,A)$ is a pair as described in Theorem~\ref{Thm equality Kemperman}. Thus there exist subsets $A', A'' \subseteq A$ and a subgroup $F$ of $G$ such that all four properties are satisfied. In particular, $(A',A'')$ is an elementary critical pair  and each is contained in a $F$-coset.

   If $0 \in A\backslash A'$, then by property (iii), $F = 0 +F \subseteq A\backslash A' \subseteq A$. Since $A \cap A^{-1} = \{0\}$ and $|F| \geq 2$ is a subgroup, this is impossible. Thus, $0 \in A'$ and similarly $0 \in A''$. Hence, both $A'$ and $A''$ are subsets of $0+F=F$. By property (ii), every element in $A \cap F$ is an element of $A'$ and $A''$. Thus, $A'=A''=A \cap F$ implying $A\backslash A'=A\backslash A''=A \backslash F$.

   By property (iii), $A\backslash A'$ is of the form $\{X+F \mid X \subseteq G\slash F, \text{ where } 0 \not\in X\}$. By property (iv), we have $$|(X+\{0\})+(X+\{0\})|= 2|X+\{0\}|-1$$ in $G\slash F$. From property (ii), we obtain that $(X\cup \{0\}) \cap (X^{-1}\cup\{0\} = \{0\}$ in $G \slash F$. Therefore, $|(X+X) \cup X|=2|X|$ and $X$ is the connection set of a Seymour orientation in the abelian group $G \slash F$. By induction, the Cayley graph corresponding to $X$ can be written as the (possibly repeated) lexicographical products of empty graphs, $k$-th powers of directed cycles and regular tournaments.

    The pair $(A',A')$ is an elementary critical pair and we know $0 \in A'$. We know look at the four cases following from the definition of elementary critical pairs. Option (a) implies that $A' = \{0\}$, thus the Cayley graph $\Gamma(F, A'-\{0\})$ is the empty graph. Option (b) together with $0 \in A'$ implies that $A' = \{0,d, \ldots, kd\}$ for some $k$ where the order of $d$ is at least $2k-1$. Hence, the Cayley graph $\Gamma(F, A'-\{0\})$ is the disjoint union of a  $k$-th power of a directed cycle.
    Option (c) together with $0 \in A'$ implies $A' \subseteq H$ for some subgroup $H$ such that $|A'|= \frac{|H|+1}{2}$ and $A'+A' =H$. Since $A' \cap A'^{-1}= \{0\}$, we obtain that the Cayley graph $\Gamma(F, A'-\{0\})$ is a disjoint union of $|F/H|$ regular tournaments. Option (d) cannot happen since $0$ can be uniquely written as $0+0$.

   Now we can write $A= (X+F) \cup A'$, where $A' \subseteq F$ and $0 \in A'$. We observe from the definition of the lexicographic product, the Cayley graph of $\Gamma(G,A-\{0\})$ is the lexicographic product of  $\Gamma(G \slash F,X)$ with  $\Gamma(F,A'-\{0\})$. The statement now follows.
\end{proof}

It is remarkable that this classification is completely combinatorial and does not depend the exact abelian group of the Cayley graphs. Therefore, one can ask of this classification holds for a larger class of Seymour orientations, see the following conjecture: 

\begin{Con}\label{conjecture cayley}
    Every Seymour Cayley orientation can be constructed by taking (possibly repeated) lexicographic products of empty graphs, the $k$-th power of a directed cycles, and regular tournaments.
\end{Con}
We can ask if it for all vertex-transitive Seymour orientations.

\begin{Prob}
   Every vertex-transitive Seymour orientation can be constructed by taking (possibly repeated) lexicographic products of empty graphs, the $k$-th power of a directed cycles, and regular tournaments.
\end{Prob}

DeVos~\cite{devos2013structure} extended Kemperman's result to non-abelian groups, which suggests a possible approach to (dis)prove Conjecture \ref{conjecture cayley}.
Specifically, DeVos characterizes all maximal critical pairs up to similarity. Observe that a pair $(A,B)$ is critical if and only if $(gA,B)$ or $(A^{-1},B^{-1})$ is. These pairs are called similar. Also, note that the pair $(S_1,S_1)$ is not necessarily maximal. Therefore, to determine all connecting sets $S$ of Seymour orientations, one should look at which 
pairs of sets in his characterization are similar to a pair that contains a critical pair $(S_1,S_1)$ such that $S_1 \cap S_1^{-1}=\{1\}$. Applying this characterization to our problem appears to necessitate a highly technical analysis beyond the scope of the present paper.

However, as a proof of concept, we can look at which trios described in Theorem 2.3 in~\cite{devos2013structure} contain two copies of a set $S_1$ satisfying $S_1 \cap S_1^{-1} = \{1\}$. Let $\Phi_1,\ldots,\Phi_m$ be the sequence such that $\Phi_{i}$ lies in some group $G_{i}$ and where $\Phi_i$ is a `continuation' of $\Phi_{i-1}$. Since $S_1 \cap S_1^{-1} = \{1\}$, every $\Phi_{i}$ contains twice a set $S_1 \cap G_i=X_i$ satisfying $X_i \cap X_{i}^{-1} = \{1\}$. Moreover, $X_{i-1}-X_i$ the union of cosets of $G_i$ in $G_{i-1}$. Thus, the Cayley graph $\Gamma(G_{i-1},X_{i-1}-\{1\})$ is the lexicographic product of $\Gamma(G_{i-1} \slash G_i , (X_{i-1} \backslash X_i)\slash G_i)$ with $\Gamma(G_i,X_i-\{1\})$.

Note that the two sets $X_i$ have the same size. Thus, any (impure) beat corresponds to a regular tournament $(|B|=|C|)$ or to a disjoint union of smaller graphs if $|A|=|B|< G_{i+1}$. Similarly, any (impure) chord corresponds to a $k$-th power of a directed cycle as both $A-X_{i+1}$ and $B-X_{i+1}$ are geometric sets. Moreover, $X_i \cap X_i^{-1}=\{1\}$ implies that $\Phi_i$ cannot be an impure dihedral chord. Lastly, all sporadic cases for $\Phi_m$ are not possible if we assume that two sets are equal to $X_m$ satisfying $X_m \cap X_m^{-1}= \{1\}$.

\section{Discussion}

In this paper, we have given some examples and general methods to construct Seymour-tight and Sullivan-tight orientations. We used these methods to construct special putative counterexamples to these conjectures. Additionally, we classified all Seymour-tight Cayley orientations of abelian groups. A natural goal would be to obtain a (partial) classification of general Seymour-tight orientations or to establish further structural properties. Progress towards this might correlate with significant progress on Seymour's second neighbourhood conjecture itself.

In addition to Conjecture~\ref{conjecture cayley}, another class of highly symmetric digraphs where the classification of Seymour-tight graphs may be tractable is the class of distance transitive orientations.
Lam~\cite{lam1980distance} observed that directed cycles and Paley tournaments are distance transitive digraphs. Moreover, he proved that the lexicographic product of a distance transitive graph with the empty graphs gives a new distance transitive graph. Note that these graphs are all Seymour-tight orientations. Bannai, Cameron and Kahn~\cite{bannai1981nonexistence} proved that there are no other distance transitive digraphs of odd girth. We conjecture the following.
\begin{Con}
Every distance transitive digraph is a Seymour-tight orientation.
\end{Con}
This might be an interesting intermediate step towards classifying all distance transitive digraphs (of even girth).

We have seen that the basic examples of Seymour-tight orientations, namely the $k$-th power of a directed cycle and regular tournaments, exhibit symmetry and regularity. However, using generalized lexicographic products, we can also construct many strongly connected  Seymour-tight orientations that do not have symmetry or regularity, see Figure~\ref{fig: not reg example 2x}. 


Another question is whether the converse, which is the orientation where all arcs are reversed, of a Seymour-tight orientation is again a Seymour-tight orientation. This is not true for all (strongly connected) Seymour-tight orientations. As an example, take $G$ to be $\dir{C_3}$ with an extra vertex $u$ that has one outgoing arc pointing towards a vertex $v$ on this cycle. Then the converse of $G$ is not a Seymour-tight orientation. The vertex $v$ has namely two in-neighbours in $G$, but only one vertex in its second in-neighbourhood. By taking the lexicographic product of a directed cycle $D$ with $G$, we obtain a strongly connected Seymour-tight orientation whose converse is not Seymour-tight. 

On the positive side, this converse-invariance holds if we impose additional symmetry conditions, namely, the converse of any vertex-transitive Seymour-tight orientation is Seymour-tight. Indeed, for any vertex-transitive orientation $O$, there exists $k$ and $m$ such that $|\NGi[+]{1}{O}{v}|=|\NGi[-]{1}{O}{v}|=k$ and $|\NGi[+]{2}{O}{v}|=|\NGi[-]{2}{O}{v}|=m$. If $O$ is also Seymour-tight, then $k =|\NGi[+]{1}{O}{v}|=|\NGi[+]{2}{O}{v}|=m$ and thus also $|\NGi[-]{1}{O}{v}|= |\NGi[-]{2}{O}{v}|$. This naturally leads to the following question: is the converse of a Seymour-tight orientation also Seymour-tight, under the weaker assumption that every vertex has in- and out-degree $k$? We conjecture that this is the case.

\begin{Con}
    Let $O$ be an orientation such that $$|\NGi[-]{1}{O}{v}|=|\NGi[+]{1}{O}{v}|=|\NGi[+]{2}{O}{v}|=k $$ for all $v \in O$. Then it also holds that $|\NGi[-]{2}{O}{v}|=k$ for all $v \in O$.
\end{Con}
Note that the conjecture holds for $k=1$ and $k=2$, by Lemmas~\ref{lem:out-degree1} and~\ref{lem:out-degree2}. 
One could even pose a stronger question: does the converse property still hold if we relax the uniformity of the constraint from the parameter $k$?
\begin{Prob}\label{open problem Eulerian}
 If $O$ is a Seymour-tight orientation and an Eulerian orientation, is then the converse orientation of $O$ also Seymour-tight? 
 \end{Prob} 

Note that this statement is true for Eulerian orientations of regular tournaments, powers of directed cycles, empty graphs and their (repeated) generalized lexicographic products. Moreover, the following construction preserves Seymour-tightness under the converse operation: take a regular tournament which contains some small regular tournament on which all other vertices are uniform, and then replace that small regular tournament with some other Seymour-tight orientation $S$ satisfying $|\NGi[-]{1}{S}{v}|=|\NGi[+]{1}{S}{v}|=|\NGi[+]{2}{S}{v}|=|\NGi[-]{2}{S}{v}|$ for all $v$. 

Seymour’s second neighbourhood conjecture in the  context of Eulerian digraphs has been studied by Cary~\cite{Cary2019}, who showed that if an Eulerian digraph $G$ admits a simple cycle partition, then Seymour’s second neighbourhood conjecture holds for $G$.

\subsubsection*{Acknowledgements}
RK was partially supported by the Dutch Research Council (NWO) grant OCENW.M20.009 and the Gravitation Programme NETWORKS (024.002.003) of the Dutch Ministry of Education, Culture and Science (OCW).

\subsubsection*{Open access statement}
For the purpose of open access, a CC BY public copyright license is applied to any Author Accepted Manuscript (AAM) arising from this submission.

\bibliographystyle{plain}
\bibliography{Seymour}

\end{document}